\documentclass[12pt,oneside,a4paper]{amsart}

\usepackage[utf8]{inputenc}
\usepackage{thmtools}
\usepackage{mathtools}
\usepackage{amssymb}

\usepackage[english]{babel}
\usepackage{enumerate}
\usepackage[fixlanguage]{babelbib}

\usepackage{mathrsfs}								%\mathscr command; http://www.math.washington.edu/~lee/Courses/typesetting-script.pdf
\usepackage{bm}                                     % Bold all the things in math mode 

\usepackage{nameref}
\usepackage{hyperref}
\usepackage[capitalize]{cleveref}
\usepackage{tikz}
\usetikzlibrary{cd}

\selectbiblanguage{english}

\declaretheorem[numberwithin=section]{theorem}
\declaretheorem[numbered=no,name=Theorem]{theorem*}
\declaretheorem[numberlike=theorem]{lemma}

\declaretheorem[numberlike=theorem]{proposition}

\declaretheorem[style=definition,numberlike=theorem]{definition}
\declaretheorem[style=definition,numberlike=theorem]{example}
\declaretheorem[style=remark,numberlike=theorem]{remark}

\newcommand{\RR}{\mathbb{R}}

\DeclarePairedDelimiter{\abs}{\lvert}{\rvert}
\DeclarePairedDelimiter{\norm}{\lVert}{\rVert}
\newcommand{\ad}{\operatorname{ad}}

\title[Lifts to central extensions]{Smooth contact lifts to central extensions of Carnot groups}
\date{August 20, 2025}
\author{Eero Hakavuori}
\address{Department of Mathematics and Statistics, P. O. Box 68 (Pietari Kalmin katu 5), FI-00014 University of Helsinki, Finland.}
\email{eero.hakavuori@helsinki.fi}
\author{Susanna Heikkilä}
\address{Department of Mathematics and Statistics, P.O. Box 35 (MaD), FI-40014 University of Jyväskylä, Finland.}
\email{susanna.a.heikkila@jyu.fi}
\author{Toni Ikonen}
\address{Department of Mathematics, University of Fribourg, Chemin du Musée 23, 1700 Fribourg, Switzerland.}
\email{toni.ikonen@unifr.ch}

\keywords{Carnot groups, central extensions, cocycles, contact lift, contact maps, Pansu differential, Pansu pullback, Rumin complex}
\subjclass[2020]{Primary: 35R03. Secondary: 20F65, 22E30, 53C17, 53C24, 57T10}

\begin{document}
\maketitle
\begin{abstract}
We consider the existence problem of lifting a smooth contact map between Carnot groups to a smooth contact map between central extensions of the original groups. Our main result is a necessary and sufficient criterion formulated using the pullback of any de Rham potential of the codomain central extension 2-cocycle: the Rumin differential of the pullback is in a linear image of the domain central extension 2-cocycle. We also show a necessary criterion using the Pansu pullback: the Pansu pullback of the codomain central extension 2-cocycle and a linear image of the domain central extension 2-cocycle are in the same Lie algebra cohomology class. We prove that the latter criterion is sufficient if the domain group is Lipschitz 1-connected, or if the pullback has maximal weight among Lie algebra $2$-cohomology classes.

\end{abstract}

\section{Introduction}

\subsection{Contact equations and rigidity}
A $C^1$-map $f\colon  G_1 \supset  U_1 \to G_2$ between Carnot groups is a \emph{contact map} if it preserves the horizontal distribution. A contact map is \emph{smooth} if it is $C^{\infty}$-regular. We do not require that contact maps have full rank, nor that they are homeomorphisms.

The contact condition leads to highly non-linear \emph{contact equations}. This makes the existence theory for contact maps delicate and leads to rigidity phenomena on Carnot groups. The rigidity is highlighted by a theorem by Pansu \cite{Pansu-1989-metriques_de_carnoth_caratheodory}: if there exists a full-rank contact diffeomorphism (or, more generally, a quasisymmetry) between Carnot groups, then the Carnot groups are isomorphic. Various authors have contributed to related rigidity results, see for instance \cite{Ambrosio-Kirchheim-2000-rectifiable-sets,magnani-2004-unrectifiability-and-rigidity-in-stratified-groups,Wenger-Young-2014-lipschitz-homotopy-groups-heisenberg} for non-full rank and \cite{capogna-cowling-2006-conformality-and-q-harmonicity,ottazzi-warhust-2011-contact-and-1-QC-on-carnot,Kleiner-Muller-Xie-2021-nonrigid_Carnot_groups} for full rank rigidity results.

Despite these rigidity results, there are powerful tools to construct contact maps in special settings, see e.g.\ \cite{koranyi-reimann-1995-foundations-for-the-theory-of-QC-maps-in-the-heisenberg-group} and \cite{gromov-1996-carnot-from-within}. The lifting problem provides a further tool to generate contact maps.

\subsection{Lifting problem}
Let $ V_1\to G_1\overset{ \pi_1}{\to} H_1$ and $ V_2\to G_2\overset{ \pi_2}{\to} H_2$ be central extensions of Carnot groups. Let $ U_1\subset H_1$ and $ \tilde{U}_1\subset G_1$ be open sets such that $ \pi_1( \tilde{U}_1)= U_1$. A map $ F\colon  \tilde{U}_1\to G_2$ is a \emph{lift} of $ f\colon  U_1\to H_2$ if $ \pi_2\circ  F =  f\circ \pi_1$. That is, if the diagram
\begin{center}
    \begin{tikzcd}
         \tilde{U}_1\rar{ F}\dar{ \pi_1}&  G_2\dar{ \pi_2}\\
         U_1\rar{ f}&  H_2
    \end{tikzcd}
\end{center}
commutes. If both $ f$ and $ F$ are contact maps, we say that $ F$ is a \emph{contact lift} of $ f$. The \emph{lifting problem} asks for necessary and sufficient conditions on the existence of a contact lift $ F$ of $ f$.

In this manuscript, we find analytic and algebraic characterizations of the lifting problem.

\subsection{Existence of contact lifts to central extensions}\label{sec-equivalent-conditions}
Central extensions can be understood in terms of $2$-cocycles, i.e., closed left-invariant $2$-forms which determine the Carnot structure of the central extension. Therefore the lifting problem is linked to the associated cocycles. To formulate our main result, let $ V_1\to G_1\overset{ \pi_1}{\to} H_1$ and $ V_2\to G_2\overset{ \pi_2}{\to} H_2$ be fixed central extensions of Carnot groups given by $2$-cocycles $ \rho_1$ and $ \rho_2$, see \cref{sec-preliminary-central-extension} for details.

Consider any smooth potential for the 2-cocycle $ \rho_2$, i.e., a 1-form $ \alpha_2\in\Omega^1( H_2)$ such that $d \alpha_2= \rho_2$. A natural path lifting criterion for the existence of a contact lift can be rephrased in terms of $( f\circ \pi_1)^* \alpha_2$ and the Rumin differential $d_c$ in $G_1$. By parsing $d_c$ through the central extension $ V_1\to G_1\to H_1$, we find our main characterization theorem.
\begin{theorem}\label{thm-lift-iff-potential-pullback-rumin-differential}
    When $ U_1 \subset  H_1$ is a simply connected domain, a smooth contact map $ f\colon H_1\supset U_1\to H_2$ admits a smooth contact lift $ F\colon G_1\supset \pi_1^{-1}( U_1)\to G_2$ if and only if there exists a linear map $L\colon V_1\to V_2$ such that $d_c f^* \alpha_2 = L\circ \pi_{E_0} \rho_1$.
\end{theorem}

Here $\pi_{E_0} \colon \Omega^{2}(  U_1;  V_1 ) \rightarrow E_0^{2}(  U_1;  V_1 )$ is the orthogonal projection to the Rumin complex $E_0(  U_1;  V_1 )$ acting componentwise on vector-valued forms, see \cref{sec-rumin-complex}.
We also note that $d_c  f^{*}(  \alpha_2' -  \alpha_2 ) = 0$ if $ \alpha_2'$ and $ \alpha_2$ are different potentials for $ \rho_2$ so the conclusion of \cref{thm-lift-iff-potential-pullback-rumin-differential} is independent of the specific potential we use.

The existence of lifts also poses a constraint that does not require computation of the Rumin differential but instead relates the Lie algebra cohomology classes of the extension cocycles via the Pansu pullback $ { f}_P^*$ defined in \cite{Kleiner-Muller-Xie-2021-pansu_pullback}.

\begin{theorem}\label{thm-if-lift-cohomology-constant}
If $ U_1 \subset  H_1$ is a domain and a smooth contact map $ f\colon H_1\supset U_1\to H_2$ admits a smooth contact lift $ F\colon G_1\supset \pi_1^{-1}( U_1)\to G_2$, then $ { f}_P^* \rho_2 =  \varphi \circ  \rho_1 + d_0 \omega$ for some graded linear map $ \varphi \colon  V_1 \rightarrow  V_2$ and $ \omega \in \Omega^{1}(  U_1;  V_2 )$.
\end{theorem}
Here $d_0 \omega$ is the Lie algebra differential (also known as the Chevalley--Eilenberg differential) of the form $ \omega$, see \cref{sec:left-invariant-d0-cocycle}.  In particular, $ { f}_P^* \rho_2 =  \varphi \circ  \rho_1 + d_0 \omega$ implies that the Lie algebra cohomology class of $ { f}_P^* \rho_2$ is left-invariant.

\subsection{Sufficient conditions for existence of contact lifts}
As in \cref{sec-equivalent-conditions}, let $ V_1\to G_1\overset{ \pi_1}{\to} H_1$ and $ V_2\to G_2\overset{ \pi_2}{\to} H_2$ be central extensions of Carnot groups given by $2$-cocycles $ \rho_1$ and $ \rho_2$, respectively. We present two positive results where the criterion of \cref{thm-if-lift-cohomology-constant} is, in fact, both necessary and sufficient.

The first is when the domain group is Lipschitz $1$-connected.

\begin{theorem}\label{thm-trivial-homotopy-implies-lift}
    Suppose that $ { f}_P^* \rho_2 =  \varphi \circ  \rho_1 + d_0 \omega$ for a graded linear map $ \varphi \colon  V_1 \rightarrow  V_2$ and a 1-form $ \omega \in \Omega^{1}(  U_1;  V_2 )$.
    If $ H_1$ is Lipschitz $1$-connected and $ U_1$ is simply connected, then $ f$ admits a smooth contact lift $ F\colon G_1\supset \pi_1^{-1}( U_1)\to G_2$.
\end{theorem} 
    The strategy of the proof is to convert the condition $ { f}_P^* \rho_2 =  \varphi \circ  \rho_1 + d_0 \omega$ on 2-cocycles to the condition in \Cref{thm-lift-iff-potential-pullback-rumin-differential} using homotopies. Lipschitz $1$-connectivity gives control on the diameter  of the area sweeped out by the homotopy. This is crucial when $ U_1$ is a proper subdomain. Recall that $ H_1$ is Lipschitz $1$-connected if there exists $\lambda > 0$ such that, for every $L > 0$, every $L$-Lipschitz map $\gamma \colon \mathbb{S}^1 \rightarrow  H_1$ from the circle $\mathbb{S}^1$ admits a $\lambda L$-Lipschitz extension $u \colon \mathbb{D} \rightarrow  H_1$ to the disk $\mathbb{D}$.

    The Carnot groups $ H_1$ to which \cref{thm-trivial-homotopy-implies-lift} applies include the Euclidean spaces, the Heisenberg groups $\mathbb{H}_n$ for $n \geq 2$ \cite{Allcock:98}, the Allcock groups \cite{Magnani-2010-contact-equations}, and jet spaces $\mathcal{J}^{k}( \mathbb{R}^n )$ for $k \geq 1$ and $n \geq 2$ \cite{Wen:You:10}, and their direct products. 

    The second result is when the 2-cocycle on the target space has large weight in comparison to the non-trivial $2$-cohomology classes in the domain.
\begin{theorem}\label{thm-max-weight-implies-lift}
    Suppose that $ { f}_P^* \rho_2 =  \varphi \circ  \rho_1 + d_0 \omega$ for a graded linear map $ \varphi \colon  V_1 \rightarrow  V_2$ and a 1-form $ \omega \in \Omega^{1}(  U_1;  V_2 )$.
    If the weight of $ \rho_2$ is equal to or greater than the maximal weight of non-trivial $2$-cocycles in the Rumin complex $E_0$ of $ H_1$ and $ U_1$ is simply connected, then $ f$ admits a smooth contact lift $ F\colon G_1\supset \pi_1^{-1}( U_1)\to G_2$.
\end{theorem}

    \cref{thm-max-weight-implies-lift} utilizes the recent center-of-mass mollification due to Kleiner--Müller--Xie \cite{Kleiner-Muller-Xie-2022-convolution_trick} and standard properties of the Rumin complex. The mollification is used to compensate the fact that in general the Pansu pullback does not commute with the exterior or the Rumin differential.

    \Cref{thm-max-weight-implies-lift} can be applied, for instance, to analyze lifts of contact maps between real filiform groups. We recall that when the step is at least two, the filiform group is not Lipschitz $1$-connected. See \Cref{rem-factorization-of-lifts-filiform} for further discussion.

\subsection{Lifting problem in the literature}
In general, the geometry of both the domain and codomain strongly affect the lifting problem. However, when the domain is one-dimensional, there are no obstructions to existence of contact lifts, and the lifts are well-understood, e.g., by control theory (see \cite[Section 8]{Agrachev-Barilari-Boscain-2020-a-comprehensive-introduction}). When the domain is higher-dimensional, only special cases of the lifting problem have been extensively studied in the literature.

Allcock's seminal contribution in \cite{Allcock:98} to the Plateau problem on higher Heisenberg groups was obtained by lifting a map $\mathbb{D} \rightarrow \mathbb{R}^{2n}$ satisfying a Lagrangian condition to a Legendrian (contact) map $\mathbb{D} \rightarrow \mathbb{H}_n$. The lifting approach has been applied to the study of minimal Lagrangian surfaces \cite{Schoen:Wolfson:01} and in the context of Allcock groups \cite{Magnani-2010-contact-equations}.

The lifting problem has also been used to generate quasisymmetric mappings between Carnot groups. Capogna--Tang \cite{Capogna-Tang-1995-quasiconformal_mappings_heisenberg} studied lifts of symplectomorphisms $\mathbb{R}^{2n} \rightarrow \mathbb{R}^{2n}$ to full-rank contact maps from $\mathbb{H}_n$ onto itself. Later Warhust \cite{War-contact-qc-real-filiform} (in the smooth context) and Xie \cite{Xie-2015-quasiconformal_maps_on_filiform_groups} (without smoothness assumptions) classified quasisymmetric automorphisms of a real filiform Carnot group of step at least three; such a map can be understood as an iterated lift of a bi-Lipschitz map $\mathbb{R}^2 \rightarrow \mathbb{R}^2$. The result is false in step two by the Korányi--Reimann examples \cite{koranyi-reimann-1985-qc-maps-on-the-heisenberg-group,koranyi-reimann-1995-foundations-for-the-theory-of-QC-maps-in-the-heisenberg-group}.

We also mention that Balogh--Hoefer-Isennegger--Tyson \cite{Balogh-Hoefer-Isenegger-Tyson-2006-horizonta_fractals_heisenberg} considered contact lifts of Lipschitz maps $\mathbb{R}^{2} \rightarrow \mathbb{R}^{2}$ in the context of iterated functions systems in the first Heisenberg group. Their methods led to the existence of horizontal BV surfaces in the first Heisenberg group which contrasts earlier results on the non-existence of Lipschitz surfaces by Ambrosio--Kirchheim \cite{Ambrosio-Kirchheim-2000-rectifiable-sets}, Magnani \cite{magnani-2004-unrectifiability-and-rigidity-in-stratified-groups}, and Wenger--Young \cite{Wenger-Young-2014-lipschitz-homotopy-groups-heisenberg}.

\subsection{Structure of the paper}
Sections~\ref{sec:prelim}--\ref{sec:uniqueness-etc} cover the basic setup.
In \cref{sec:prelim}, we recall preliminary notions on Carnot groups, the Rumin complex, and the Pansu pullback.
In \Cref{sec-preliminary-central-extension}, we recall basic facts about central extensions of Lie algebras and Lie groups, and adapt them to the Carnot setting.
In \cref{sec-structure-of-lifts}, we reduce the study of contact lifts to central extensions that preserve the horizontal rank.
In \cref{sec:lifts}, we characterize the existence of contact lifts in terms of horizontal path lifting, and rephrase the path lifting condition using a potential of the central extension 2-cocycle.
In \cref{sec:uniqueness-etc}, we discuss structural properties of contact lifts, including uniqueness and behavior of the lift in the fiber direction.

Sections~\ref{sec:rumin}--\ref{sec:lip1-connected} cover the proofs of our main results.
In \cref{sec:rumin}, we compute how the Rumin differentials of 1-forms transform in central extensions of Carnot groups, and prove Theorem \ref{thm-lift-iff-potential-pullback-rumin-differential}.
In \cref{sec-common-extension}, we study the connection between the Pansu pullback and lifts to central extensions, and prove \cref{thm-if-lift-cohomology-constant,thm-max-weight-implies-lift}.
In \cref{sec:lip1-connected}, we consider the Lipschitz $1$-connected setting, and prove \cref{thm-trivial-homotopy-implies-lift}.

Finally in \cref{sec:examples}, we give examples of Carnot groups and contact maps that do (not) admit contact lifts.

\textbf{Acknowledgements.}
E.H. was supported by the Research Council of Finland (grant 347964 The abnormal curves of Sub-Riemannian geometry). 
S.H. was supported by the Research Council of Finland, project number 360505.
T.I. was supported by the Swiss National Science Foundation grant 212867.
S.H. and T.I. were also supported by the Research Council of Finland, project number 332671.

\section{Preliminaries}\label{sec:prelim}
\subsection{Carnot groups}\label{sec-preliminary-carnot-group}
We recall the basic properties of Carnot groups. We refer the reader to \cite{LeDonne-2017-primer_on_carnot_groups} for further details.

A Lie algebra $ \mathfrak{g}$ is \emph{stratified} if it has a decomposition $ \mathfrak{g} =   \mathfrak{g}^{[1]} \oplus \dots \oplus  \mathfrak{g}^{[s]}$ with $[  \mathfrak{g}^{[1]},\mathfrak{g}^{[k]}] =  \mathfrak{g}^{[k+1]}$ for all $k\geq 1$, where we denote $  \mathfrak{g}^{[k]} = \{0\}$ for $k \geq s+1$. 
A \emph{(sub-Riemannian) Carnot group} is a simply connected nilpotent Lie group $ G$ with a stratified Lie algebra $ \mathfrak{g}$ equipped with an inner product on the horizontal layer $  \mathfrak{g}^{[1]}$. For considerations involving the Rumin complex, we will need an inner product not only on the horizontal layer $  \mathfrak{g}^{[1]}$, but instead on the whole Lie algebra $ \mathfrak{g}$. We always equip the Lie algebra of a Carnot group with an inner product for which the layers $  \mathfrak{g}^{[1]} \oplus \dots \oplus  \mathfrak{g}^{[s]}$ are pairwise orthogonal. The \emph{homogeneous dimension} of $ G$ is $ Q_{ G} \coloneqq \sum_{ k = 1 }^{ s } k \mathrm{dim}(   \mathfrak{g}^{[k]} )$. The \emph{(horizontal) rank} of $G$ is the dimension of the horizontal layer $\mathfrak{g}^{[1]}$.

Given $g \in  G$, we denote by $L_g \colon  G \rightarrow  G$ the \emph{left-translation} $L_g(h) = gh$. We identify $T_{e} G$ with $ \mathfrak{g}$. An absolutely continuous curve $ \gamma\colon[0,1]\to G$ is \emph{horizontal} if for almost every $t\in[0,1]$ its left-trivialized derivative $(L_{ \gamma(t)}^{-1})_*\dot{ \gamma}(t)$ is contained in the horizontal layer $  \mathfrak{g}^{[1]}$. The length of the horizontal curve $ \gamma$ is
\begin{equation*}
    \ell( \gamma) = \int_0^1 \norm{(L_{ \gamma(t)}^{-1})_*\dot{ \gamma}(t)}\,dt,
\end{equation*}
where the norm is the one induced by the inner product on $  \mathfrak{g}^{[1]}$. The sub-Riemannian distance between two points $g,h\in G$ is
\begin{equation*}
    d(g,h) = \inf\{ \ell( \gamma)\;\vert\;  \gamma\colon[0,1]\to G\text{ horizontal},  \gamma(0)=g, \gamma(1)=h \}.
\end{equation*}
By construction, this distance is left-invariant.

Carnot groups also admit a one-parameter group of automorphisms, known as the dilations $ \delta_{\lambda}\colon G\to G$, $\lambda>0$. The associated Lie algebra automorphisms, also denoted $ \delta_{\lambda}\colon \mathfrak{g}\to \mathfrak{g}$, are the linear maps defined on each layer of the stratification as
\begin{equation*}
     \delta_{\lambda}(X) = \lambda^kX,\quad X\in   \mathfrak{g}^{[k]}, \quad 1 \leq k \leq s.
\end{equation*}
The sub-Riemannian distance is 1-homogeneous with respect to these dilations, i.e., $d( \delta_{\lambda}g, \delta_{\lambda}h)=\lambda d(g,h)$.

\subsection{Left-invariant forms, the Lie algebra differential, and cocycles}\label{sec:left-invariant-d0-cocycle}
Let $ \mathfrak{h}$ be a Lie algebra, $ V$ a vector space, and $ \rho\colon \bigwedge^k \mathfrak{h}\to  V$ a vector-valued $k$-form. Here, and in what follows, $ V$ is assumed to be a finite-dimensional vector space. The \emph{Lie algebra differential $ d_0 \rho$ of $ \rho$} is the vector-valued $(k+1)$-form
\begin{align*}
    & d_0  \rho(X_1,\dots,X_{k+1}) \\
    &\quad = \sum_{i<j}(-1)^{i+j}   \rho([X_i,X_j],X_1,\ldots,\hat{X}_i,\ldots,\hat{X}_j,\ldots,X_{k+1}),
\end{align*}
where $\hat{X}_i$ means that $X_i$ is omitted from the list.
We identify $\bigwedge^1  \mathfrak{h} \simeq  \mathfrak{h}$, so vector valued $1$-forms are identified with linear maps $ \mathfrak{h} \rightarrow  V$. 
A \emph{$k$-cocycle} is a $k$-form whose Lie algebra differential vanishes.

When $ H$ is a Lie group with Lie algebra $ \mathfrak{h}$, we will identify $k$-forms $ \rho\colon \bigwedge^k \mathfrak{h}\to  V$ with their left-invariant extensions to differential $k$-forms on $ H$. Under this identification, the Lie algebra differential of $ \rho\colon \bigwedge^k \mathfrak{h}\to  V$ coincides with the exterior differential of the left-invariant differential form. In particular, a 2-cocycle $ \rho\colon \bigwedge^2 \mathfrak{h}\to  V$ can be viewed as a vector-valued left-invariant 2-form $ \rho\in\Omega^2( H; V)$ whose exterior differential vanishes.

\subsection{Rumin complex}\label{sec-rumin-complex}
We recall the basic properties of the Rumin complex introduced in \cite{Rumin-1999-differential_geometry_on_cc_spaces}. See \cite{Rumin-around-heat-decay} or \cite{Franchi-Tripaldi-2015-diff_forms_in_carnot_after_rumin} for more comprehensive presentations.

Let $ G$ be a Carnot group, and recall that there exists an inner product on $ \mathfrak{g}$ such that the layers $  \mathfrak{g}^{[k]}$ of the stratification are pairwise orthogonal. Fixing a basis and using duality, we obtain an inner product on $k$-forms $\bigwedge^k \mathfrak{g}^*$ for each $k\geq 1$. Using the inner product, we may consider the orthogonal projection $\pi_{ \operatorname{im}(d_0) }$ onto $\operatorname{im}(d_0)$ and the orthogonal complement $\operatorname{ker}(d_0)^{\perp}$. We let $d_0^{-1}$ denote the extension of the left-inverse of $d_0|_{ \operatorname{ker}(d_0)^{\perp} }$ satisfying $d_0^{-1} \circ \pi_{ \operatorname{im}(d_0) } = d_{0}^{-1}$.

Let $ U\subset G$ be a domain. The stratification on $ \mathfrak{g}$ induces a grading by weights also on the differential forms $\Omega^*( U)$, see \cite{Franchi-Tripaldi-2015-diff_forms_in_carnot_after_rumin} for the specifics. 
The relevant basic properties we recall are as follows. The weight of a horizontal 1-form is one. Representing a non-zero form in a left-invariant basis of pure weight forms, the weight of the form is the minimum of the weights of the non-zero terms. For a form of pure weight $w$, the Hodge star has weight $ Q_{ G}-w$. A $ V$-valued form has a componentwise representation relative to a basis of $ V$, and its weight is the minimum of the componentwise weights.

The weight preserving component of the de Rham differential $d$ is the $\mathcal{C}^\infty( U)$-linear extension of the Lie algebra differential $d_0$, which we also denote by $d_0$. We also extend $\pi_{ \mathrm{im}(d_0) }$ and $d_0^{-1}$ $\mathcal{C}^{\infty}(  U )$-linearly from $\bigwedge^{*} \mathfrak{g}^{*}$ to all of $\Omega^*( U)$. Consequently, $\ker(d_0)$, $\operatorname{im}(d_0)$, and their orthogonal complements have bases as $\mathcal{C}^{\infty}( U)$-modules consisting of pure weight left-invariant forms. Moreover, the orthogonal projections to these submodules do not decrease weight.

There is a subcomplex of $\Omega^*( U)$ defined by
\begin{equation*}
    E( U) = \mathrm{ker}( d_0^{-1} ) \cap \mathrm{ker}( d_0^{-1} d ),
\end{equation*}
with a projection operator $\pi_{E} \colon \Omega^{*}(  U ) \rightarrow E( U)$ such that $d\pi_E=\pi_Ed$, constructed as follows.

We let $D = d_0^{-1}( d - d_0 )$ and $P = \sum_{ k \geq 0 } (-1)^{k} D^{k}$. Here the summation is, in fact, finite, since the operator $D$ strictly increases the weight of forms, but by nilpotency of $ \mathfrak{g}$, there is an upper bound on the possible weights. The projection operator $\pi_E$ is defined by
\begin{equation}\label{eq-pi-E-formula}
    \pi_E = I - Pd_0^{-1}d - dPd_0^{-1}.
\end{equation}
Since each of the operators $d$, $d_0^{-1}$, and $P$ are weight non-decreasing, so is the projection $\pi_E$.

Next, we define
\begin{equation*}
    E_0( U) = \mathrm{ker}( d_0 ) \cap \mathrm{ker}( d_0^{-1} ) \subset \Omega^{*}(  U )
\end{equation*}
and let $\pi_{E_0}\colon \Omega^*( U)\to E_0( U)$ be the orthogonal projection.
The \emph{Rumin differential} is the map 
\begin{equation}\label{eq-rumin-differential}
    d_c \colon \Omega^{*}(  U ) \rightarrow E_0( U),\quad d_c\omega \coloneqq \pi_{E_0}\pi_{E}d\pi_{E_0}\omega, 
\end{equation}
and the pair $( E_0( U), d_c )$ is the \emph{Rumin complex}. 

By \cite[Theorem 1]{Rumin-1999-differential_geometry_on_cc_spaces}, we have the following result:

\begin{proposition}\label{prop:cohomology}
The pairs $( E( U), d )$ and $( E_{0}( U), d_c )$ are chain complexes. The projections $$\pi_E \colon ( E_0( U), d_c ) \rightarrow ( E( U), d )$$ and $$\pi_{E_0} \colon ( E( U), d ) \rightarrow ( E_0( U), d_c )$$ are mutually inverse chain maps. Furthermore, the cohomologies of these chain complexes are isomorphic to the de Rham cohomology.
\end{proposition}

We recall some further basic properties of the complex $E_0$ and its projection. Namely, the complex $E_0$ is invariant under the Hodge star operator, and we have $\pi_{E_0}\omega \wedge \eta = \omega \wedge \pi_{E_0}\eta = \pi_{E_0}\omega \wedge \pi_{E_0}\eta$ for $\omega \in \Omega^k( U)$ and $\eta \in \Omega^l( U)$ such that $\omega \wedge \eta$ is top-dimensional, i.e., $k+l = n$ for the topological dimension $n$ of $ U$.

On the level of integrals, the projection $\pi_{E}$ behaves similarly to $\pi_{E_0}$ with respect to the wedge product. Namely, if $\omega \in \Omega^{k}(  U )$ and $\eta \in \Omega^{l}_{c}(  U )$ are such that $\omega \wedge \eta$ is top-dimensional, then
\begin{align*}
    \int_{  U } \pi_{E}\omega \wedge \eta
    =
    \int_{  U } \omega \wedge \pi_{E}\eta
    =
    \int_{  U } \pi_{E}\omega \wedge \pi_{E}\eta.
\end{align*}

We use the notation $\partial_c \eta = ( -1 )^{k+1} d_c\eta$ (resp. $\partial \eta = (-1)^{k+1} d\eta$) when $\eta$ is a form of degree $n-k-1$. Then it holds that
\begin{align*}
    \int_{  U } d_c\omega \wedge \eta
    =
    \int_{  U } \omega \wedge \partial_c\eta
\end{align*}
whenever $\omega \in \Omega^{k}(  U )$ and $\eta \in \Omega^{n-k-1}_{c}(  U )$. For details, see \cite[Section~2 and Proposition~2.8]{Rumin-around-heat-decay}.

\subsection{Gradings and graded maps}\label{sec-preliminary-gradings}
Recall that a \emph{grading} of a vector space $ V$ is a direct sum decomposition $ V= \bigoplus_{ j \geq 1 }  V^{[j]}$. A linear map $L \colon  V_1 \to  V_2$ between graded vector spaces is \emph{graded} if $L(  V_1^{[k]})\subset V_2^{[k]}$ for $k\geq 1$. 

A Lie algebra $ \mathfrak{g}$ is \emph{graded} if it has a grading compatible with the Lie bracket. That is, $[   \mathfrak{g}^{[j]}, \mathfrak{g}^{[k]} ] \subset  \mathfrak{g}^{[j+k]}$ for $k, j \geq 1$. A stratification is a particular case of a grading. Moreover, a graded vector space can be considered a graded Lie algebra when equipped with the trivial Lie bracket.

A simply connected Lie group is graded if its Lie algebra is graded. A Lie group homomorphism $L \colon  G_1 \to  G_2$ between graded simply connected Lie groups is \emph{graded} if the corresponding Lie algebra homomorphism $L_* \colon  \mathfrak{g}_1 \to  \mathfrak{g}_2$ is graded.

\subsection{Pansu differential and pullback}\label{sec-preliminary-pansu-pullback}
Let $ G_1$ and $ G_2$ be two Carnot groups.

Suppose that $ f \colon  U_1 \to  G_2$ is a (locally) Lipschitz map on a domain $ U_1\subset G_1$. By the Pansu--Rademacher Theorem \cite{Pansu-1989-metriques_de_carnoth_caratheodory}, at almost every $g \in  U_1$, the map $ f$ has a \emph{Pansu differential} 
\begin{equation*}
     d_P{f}(g)= \lim_{ \lambda \to 0^{+} }  \delta_{1/\lambda} \circ L^{-1}_{ f(g)}\circ  f\circ L_g\circ \delta_{\lambda}\colon G_1\to G_2,
\end{equation*}
where the convergence is uniform on compact subsets of $ G_1$. The Pansu differential, whenever it exists, is a graded Lie group homomorphism.

For a point $g\in  U_1$ where the Pansu derivative of $ f$ exists, denote 
\begin{equation*}
     D_P{ f}(g) = (L_{ f(g)})_*\circ d_P{f}(g)_*\circ (L_{g^{-1}})_*\colon T_g G_1\to T_{ f(g)} G_2.
\end{equation*}
This leads to a natural adaptation of the usual de Rham pullback by a smooth map as follows.
\begin{definition}[\cite{Kleiner-Muller-Xie-2021-pansu_pullback}]\label{def:pansupullback}
The \emph{Pansu pullback} of a smooth $k$-form $\omega\in\Omega( G_2)$ is the $k$-form $ {f}_P^*\omega$ defined by
\begin{equation*}
    ( {f}_P^*\omega)_g(X_1,\ldots,X_k) = \omega_{ f(g)}( D_P{f}(g)X_1,\ldots,  D_P{f}(g)X_k)
\end{equation*}
for $X_1,\ldots,X_k\in T_g G_1$ and $g\in U_1$.
\end{definition}
The Pansu pullback extends for vector-valued forms naturally.

\begin{remark}
If $ f$ is a graded homomorphism, then $ d_P{ f}(g) =  f$ everywhere and the Pansu pullback coincides with the classical de Rham pullback $ f^{*}$.
\end{remark}

\section{Central extensions}\label{sec-preliminary-central-extension}

\subsection{Central extension of Lie algebras and Lie groups}
In this subsection, we recall standard facts about central extensions of Lie algebras and Lie groups.

\begin{definition}\label{def-central-extension}
    Let $ \mathfrak{h}$ be a Lie algebra and $ \rho\colon \bigwedge^2 \mathfrak{h}\to  V$ a 2-cocycle with values in a vector space $ V$. Denote the Lie bracket of $ \mathfrak{h}$ by $[\cdot,\cdot]_{ \mathfrak{h}}$. 
    The \emph{central extension of $ \mathfrak{h}$ by $ \rho$} is a Lie algebra $ \mathfrak{g}$ with a direct sum decomposition $ \mathfrak{h}\oplus V$ equipped with the Lie bracket
    \begin{equation*}
        [X+A,Y+B]_{ \mathfrak{g}} = [X,Y]_{ \mathfrak{h}} +  \rho(X,Y),\quad X,Y\in \mathfrak{h},\quad A,B\in V.
    \end{equation*}
    The direct sum decomposition $ \mathfrak{g} =  \mathfrak{h}\oplus V$ induces a natural inclusion $\iota_{*} \colon  V \to  \mathfrak{g}$ and a projection $ \pi_{*} \colon  \mathfrak{g} \to  \mathfrak{h}$ which are Lie algebra homomorphisms. The central extension is denoted by $ V \overset{\iota_{*}}{\to}  \mathfrak{g} \overset{ \pi_{*}}{\to}  \mathfrak{h}$.
    
    If $ G$ and $ H$ are simply connected Lie groups with Lie algebras $ \mathfrak{g}$ and $ \mathfrak{h}$, respectively, we also refer to the induced short exact sequence $ V \overset{\iota}{\to}  G\overset{ \pi}{\to} H$ as a central extension of $ H$ by $ \rho$.
\end{definition}
    We often suppress the inclusions maps $ V \to  \mathfrak{g}$ and $ \mathfrak{h} \to  \mathfrak{g}$ to simplify notation.

\begin{remark}\label{rem:graded}
    Central extensions using a 2-cocycle are, in a sense, the converse to considering a section of a surjective Lie algebra homomorphism $ \pi_{*}\colon \mathfrak{g}\to \mathfrak{h}$ such that $\ker( \pi_{*})$ is contained in the center of $ \mathfrak{g}$. Indeed, as observed, for instance in \cite[Section~2]{deGraaf-2007-classification_of_6d_nilpotent}, if $\sigma\colon  \mathfrak{h}\to \mathfrak{g}$ is a linear section of $ \pi_{*}$, then the map $ \rho\colon\bigwedge^2 \mathfrak{h}\to\ker( \pi_{*})$ defined by
    \begin{equation}\label{eq:rem:graded:cocycle}
         \rho(X,Y) = [\sigma(X),\sigma(Y)]_{ \mathfrak{g}} - \sigma( [X,Y]_{ \mathfrak{h}} )
    \end{equation}
    is a 2-cocycle. The central extension of $ \mathfrak{h}$ by $ \rho$ is isomorphic to $ \mathfrak{g}$, with the isomorphism given by
    \begin{equation*}
         \mathfrak{h}\oplus \ker( \pi_{*})\ni X+Y\mapsto \sigma(X) + Y\in \mathfrak{g}.
    \end{equation*}
\end{remark}

\subsection{Central extension of Carnot groups}
    We mainly work with central extensions of Carnot groups. The definition of a central extension needs to be adapted to account for the additional structure. We do this in two parts --- algebraic and metric --- as follows.
\begin{definition}\label{definition-carnot-central-extension}
    A central extension $ V \overset{\iota}{\to}  G\overset{ \pi}{\to} H$ by $ \rho$ is \emph{stratified} if the following properties hold.
    \begin{enumerate}[(i)]
        \item The Lie algebras $ \mathfrak{h}$ and $ \mathfrak{g}$ are stratified and the vector space $ V$ is graded.
        \item The inclusion $\iota_{*} \colon  V \to  \mathfrak{g}$ and the projection $ \pi_{*} \colon  \mathfrak{g} \to  H$ are graded with respect to the gradings from (i).
    \end{enumerate}
    A stratified central extension $ V \overset{\iota}{\to}  G\overset{ \pi}{\to} H$ by $ \rho$ is a \emph{central extension of Carnot groups} if the following hold.
    \begin{enumerate}[(i)] \addtocounter{enumi}{2}
        \item The Lie groups $ G$ and $ H$ are Carnot groups and $ V$ is an inner product space.
        \item The inclusion $\iota_{*} \colon  V \to  \mathfrak{g}$ is an isometry and $ \pi_{*} \colon  \mathfrak{g} \to  \mathfrak{h}$ is a submetry with respect to the inner product structures from (iii).
    \end{enumerate}
\end{definition}
A linear map $ \pi_{*} \colon  \mathfrak{g} \to  \mathfrak{h}$ between inner product spaces is a \emph{submetry} if the restriction of $ \pi_{*}$ to $\operatorname{ker}( \pi_{*})^{\perp}$ is an isometry onto $ \mathfrak{h}$.

The definition of central extension of Carnot groups is motivated by the following adaptation of \Cref{rem:graded} to the Carnot setting.

\begin{remark}\label{rem:centralextension:carnot}
    Consider Carnot groups $ G$ and $ H$ together with a graded homomorphism $ \pi \colon  G \to  H$ for which $ \pi_{*} \colon  \mathfrak{g} \to  \mathfrak{h}$ is a submetry and $\ker(  \pi_{*} )$ is contained in the center of $ \mathfrak{g}$.
    
    Equip $ \mathfrak{h} \oplus \ker( \pi_{*})$ with the direct sum of the inner products on $ \mathfrak{h}$ and $\ker( \pi_{*})$.  Next, consider the isometric graded section $\sigma \colon  \mathfrak{h}\to \ker( \pi_{*})^{\perp} \subset  \mathfrak{g}$ of $ \pi_{*} \colon  \mathfrak{g} \to  \mathfrak{h}$. Then, defining $ \rho$ by \eqref{eq:rem:graded:cocycle}, we obtain a central extension $\ker( \pi_{*}) \to  \mathfrak{h} \oplus \ker( \pi_{*}) \to  \mathfrak{h}$ by $ \rho$, where
    \begin{equation*}
         \mathfrak{h}\oplus \ker( \pi_{*}) \ni X+Y\mapsto \sigma(X) + Y\in \ker( \pi_{*})^{\perp} \oplus \ker( \pi_{*}) =  \mathfrak{g}
    \end{equation*}
    is a graded isomorphism that is an isometry. In particular, the corresponding central extension of Lie groups by $ \rho$ is a central extension of Carnot groups.
\end{remark}   

We observe that the construction of the sub-Riemannian distance guarantees that $ \pi \colon  G \to  H$ is a metric submetry. That is, $ \pi(B(g,r)) = B( \pi(g),r)$ for all balls $B(g,r)\subset G$.

\subsection{Basic results about central extensions}

We start this subsection by considering Lie algebra homomorphisms in the context of central extensions.
\begin{lemma}\label{lemma-homomorphism-central-extension}
    Let $ V_1 \to  \mathfrak{g}_1\to \mathfrak{h}_1$ and $ V_2 \to  \mathfrak{g}_2 \to  \mathfrak{h}_2$ be central extensions by 2-cocycles $ \rho_1$ and $ \rho_2$, respectively. Let $L \colon  \mathfrak{h}_1 \rightarrow  \mathfrak{h}_2$ be a Lie algebra homomorphism and let $ \varphi\colon  V_1\to V_2$ be a linear map.
    \begin{enumerate}[(i)]
    \item\label{lemma-homomorphism-central-extension-existence}
    There exists a Lie algebra homomorphism $ \psi\colon  \mathfrak{g}_1\to \mathfrak{g}_2$ such that the diagram
    \begin{center}
		\begin{tikzcd}
			 V_1 \rar\dar{ \varphi}&  \mathfrak{g}_1\rar{ \pi_1}\dar{ \psi}&  \mathfrak{h}_1\dar{L}\\
			 V_2 \rar&  \mathfrak{g}_2\rar{ \pi_2}&  \mathfrak{h}_2
		\end{tikzcd}
    \end{center}
    commutes if and only if there exists a linear map $\mu \colon  \mathfrak{h}_1 \rightarrow  V_2$ such that $ \varphi\circ \rho_1 - L^{*} \rho_2 =  d_0\mu$. When such a $\mu$ exists, a homomorphism $ \psi$ is given by
    \begin{equation}\label{eq-generic-homomorphism-extension}
         \psi(X+Y) = L(X) + \Big(\mu(X)+ \varphi(Y)\Big) \in  \mathfrak{h}_2\oplus V_2= \mathfrak{g}_2
    \end{equation}
    for $X+Y \in  \mathfrak{h}_1 \oplus  V_1= \mathfrak{g}_1$.
    \item\label{lemma-homomorphism-central-extension-graded-uniqueness}
    When the central extensions are stratified and $ \varphi$ and $L$ are graded, a graded homomorphism $ \psi$ is unique up to a graded linear map $\theta \colon  \mathfrak{h}_1 \rightarrow  V_2$ with $[ \mathfrak{h}_1, \mathfrak{h}_1]\subset\ker\theta$. In particular, if $ \mathfrak{h}_{2}^{[1]}= \mathfrak{g}_{2}^{[1]}$, then the graded homomorphism $ \psi$ is unique.
    \end{enumerate}
\end{lemma}
\begin{proof}
    \eqref{lemma-homomorphism-central-extension-existence} The linear maps $ \psi$ for which the diagram in the claim commutes necessarily have the form \eqref{eq-generic-homomorphism-extension}, where $\mu \colon  \mathfrak{h}_1 \rightarrow  V_2$ is a linear map. We need to check when $ \psi$ is a Lie algebra homomorphism. 
    
    If $Z,Z' \in  \mathfrak{h}_1 \oplus  V_1$, then
    \begin{align*}
         \psi([Z,Z']_{ \mathfrak{g}_1})
        &=
         \psi( [ \pi_1(Z), \pi_1(Z')]_{ \mathfrak{h}_1} +  \rho_1(  \pi_1(Z),  \pi_1(Z')) )
        \\
        &=
        L([ \pi_1(Z),  \pi_1(Z')]_{ \mathfrak{h}_1}) + \mu([ \pi_1(Z),  \pi_1(Z')]_{ \mathfrak{h}_1})
        \\
        &\quad+  \varphi\circ  \rho_1 (  \pi_1(Z),  \pi_1(Z'))
    \end{align*}
    and
    \begin{align*}
        [  \psi(Z),  \psi(Z') ]_{ \mathfrak{g}_2}
        =
         [ L( \pi_1(Z)), L( \pi_1(Z')) ]_{ \mathfrak{h}_2} +    \rho_2 ( L( \pi_1(Z)), L( \pi_1(Z')) ).
    \end{align*}
    Since $L$ is by assumption a Lie algebra homomorphism, the $ \mathfrak{h}_2$ components of the above expressions agree. Comparing the $ V_2$ components and writing $\mu([\cdot,\cdot]_{ \mathfrak{h}_1})$ as $-d_0\mu(\cdot,\cdot)$, we observe that $ \psi$ is a Lie algebra homomorphism if and only if $-d_0\mu +  \varphi\circ  \rho_1 = L^* \rho_2$.

    \eqref{lemma-homomorphism-central-extension-graded-uniqueness} 
    Suppose $ \psi$ and $ \psi'$ are two graded homomorphisms from \eqref{lemma-homomorphism-central-extension-existence}. By \eqref{eq-generic-homomorphism-extension}, their difference is
    \begin{equation*}
         \psi(X,Y)- \psi'(X,Y) = \mu(X)-\mu'(X)\in V_2.
    \end{equation*}
    The difference $\theta=\mu-\mu'\colon \mathfrak{h}_1\to V_2$ is a graded linear map with $d_0\theta = d_0\mu-d_0\mu' = 0$, so $[ \mathfrak{h}_1, \mathfrak{h}_1]\subset\ker\theta$. If, furthermore, $ \mathfrak{h}_{2}^{[1]}= \mathfrak{g}_{2}^{[1]}$, then $ V_{2}^{[1]} = 0$. Then the graded assumption implies $ \mathfrak{h}_{1}^{[1]}\subset\ker\theta$ and thus $\theta=0$, so the graded homomorphism $ \psi$ is unique.
\end{proof}

Given a basis $v_1,\ldots,v_m$ of $ V$, we may decompose the 2-cocycle $ \rho\colon \bigwedge^2 \mathfrak{h}\to  V$ into 2-cocycles $ \rho^j\colon \bigwedge^2 \mathfrak{h}\to \RR$ such that
\begin{equation*}
     \rho(X,Y) = \sum_{j=1}^{m} \rho^j(X,Y)v_j,\quad X,Y\in  \mathfrak{h}.
\end{equation*}
We denote $ \rho=\sum_{j=1}^{m}v_j \rho^j$ for brevity.

Within the proofs of our main results, it will be convenient to assume that the cocycles $ \rho^j$ are either pairwise orthogonal in the Rumin complex, see \cref{sec-rumin-complex}, or are contained in the Rumin complex $E_0$ themselves. For this purpose, we recall a limited form of \cite[Lemma~3]{deGraaf-2007-classification_of_6d_nilpotent}:

\begin{lemma}\label{lemma-carnot-central-extension-isomorphism}
    Let $v_1,\ldots,v_m$ be a basis of $ V$ and let $ V\to  G\to H$ be a central extension by a cocycle $ \rho=\sum_{j=1}^{m}v_j \rho^j\colon \bigwedge^2 \mathfrak{h}\to  V$. Suppose that $\tilde{ \rho} = \sum_{j=1}^{m}v_j\tilde{ \rho}^j\colon \bigwedge^2 \mathfrak{h}\to  V$ is another cocycle and $ \omega^1,\ldots, \omega^m\colon \bigwedge^1 \mathfrak{h}\to \RR$ are 1-forms for which the equality
    \begin{equation*}
        \operatorname{span}\{ \rho^{1},\ldots, \rho^{m}\} = \operatorname{span}\{\tilde{ \rho}^{1}+d_0 \omega^{1},\ldots,\tilde{ \rho}^{m}+d_0 \omega^{m}\}
    \end{equation*}
    of $\mathbb{R}$-linear spans holds in the space of $2$-forms. Then a central extension $ V\to\tilde{ G}\to H$ by $\tilde{ \rho}$ is isomorphic to the central extension $ V \to  G \to  H$ by $ \rho$ via a commutative diagram
    \begin{center}
    \begin{tikzcd}
         V\dar{ \varphi}\rar& G\rar\dar{ \psi}& H\dar{\operatorname{id}}\\
         V\rar&\tilde{ G}\rar& H
    \end{tikzcd}
    \end{center}
\end{lemma}
\begin{proof}
    By the assumption on spanned subspaces, there exists an invertible matrix $A=(a_{ij})\in\RR^{m\times m}$ such that $\tilde{ \rho}^i + d_0 \omega^i = \sum_{j=1}^ma_{ij} \rho^j$. Indeed, denote $W = \operatorname{span}\{ \rho^{1},\ldots, \rho^{m}\}$ and consider the two maps $E_1,E_2\colon\RR^m\to W$ defined by $E_1(e_i) =  \rho^i$ and $E_2(e_i) = \tilde{ \rho}^i+d_0 \omega^i$, using the standard basis of $\mathbb{R}^m$. Fixing any isomorphism $L\colon \ker E_1\to\ker E_2$, there exists an isomorphism $A \colon \mathbb{R}^{m} \rightarrow \mathbb{R}^m$ with the short exact sequence of a pair
    \begin{center}
        \begin{tikzcd}
                0\rar& \ker E_1 \rar\dar{L}& \RR^m\dar{A} \rar{E_1}& W\dar{\operatorname{id}}\rar& 0 \\
                0\rar& \ker E_2 \rar& \RR^m \rar{E_2}& W\rar& 0 
        \end{tikzcd}
    \end{center}
    If we let $ \varphi\colon V\to V$ be the linear map whose matrix in the basis $v_1,\ldots,v_m$ is $A$, we have $ \varphi\circ \rho - \tilde{ \rho} = d_0 \omega$ for $ \omega=\sum_{j=1}^m v_j  \omega^j$ and the map 
    \begin{equation*}
         \psi\colon \mathfrak{g}\to\tilde{ \mathfrak{g}},\quad  \psi(X+Y) = X + ( \omega(X) +  \varphi(Y)), \quad X\in \mathfrak{h}, \; Y\in V
    \end{equation*}
    is a Lie algebra homomorphism by \cref{lemma-homomorphism-central-extension}. Since $ \varphi$ is invertible, $ \psi$ is an isomorphism of Lie groups.
\end{proof}

\begin{remark}\label{remark-carnot-central-extension-isomorphism-tilting}
In our applications of \cref{lemma-carnot-central-extension-isomorphism}, the extension $ V \overset{\iota}{\to}  G\overset{ \pi}{\to} H$ will be a central extension of Carnot groups in which case we will define a Carnot structure on $\tilde{ G}$ through the isomorphism $ \psi \colon  G \to \tilde{ G}$ given by the lemma. More precisely, the Lie algebra $\tilde{ \mathfrak{g}}= \mathfrak{h} \oplus  V$ has a projection $\tilde{ \pi}_{*} \colon \tilde{ \mathfrak{g}} \to  \mathfrak{h}$. Consider $\iota_{*} \colon  V \to  \mathfrak{g}$ and the right-inverse $\sigma \colon  \mathfrak{h} \to \iota_{*}( V)^{\perp}$ of $ \pi_{*} \colon  \mathfrak{g} \to  \mathfrak{h}$. Define an inner product and grading on $\tilde{ \mathfrak{g}}$ for which $\tilde{\iota}_{*} \coloneqq  \psi \circ \iota_{*} \colon  V \to \tilde{ \mathfrak{g}}$ and $\tilde{\sigma} \coloneqq  \psi \circ \sigma \colon  \mathfrak{h} \to \tilde{ \mathfrak{g}}$ are graded isometries with orthogonal images. Then $\tilde{ \pi}_{*}$ is a submetry with $\ker( \tilde{ \pi}_{*} )^{\perp} = \operatorname{im}( \tilde{\sigma} )$, and the central extension $ V \overset{\tilde{\iota}}{\to} \tilde{ G} \overset{\tilde{ \pi}}{\to}  H$ by $\tilde{ \rho}$ is a central extension of Carnot groups.
\end{remark}

The following lemma allows us to generate further central extensions of Carnot groups from graded linear maps. We use the construction in \Cref{lemma-existence-of-lift-in-intermediate-extension-implies-lift} as a simplifying tool for the proofs of our main results.
\begin{lemma}\label{lemm:centralextension:morphism}
    Let $ V_1\to G_1\to H_1$ and $ V_2\to G_2\to H_2$ be central extensions of Carnot groups by $2$-cocycles $ \rho_1$ and $ \rho_2$, respectively, and let $ \varphi \colon  V_1 \to  V_2$ be a graded linear map. Then $ \hat{\rho}_1\coloneqq \varphi\circ \rho_1$ is a $2$-cocycle and defines a central extension $ \operatorname{im}( \varphi) \to  \hat{\mathfrak{g}}_1 \to  \mathfrak{h}_1$ by $ \hat{\rho}_1$. The central extension $\operatorname{im}( \varphi) \to  \hat{G}_1\to H_1$ by $ \hat{\rho}_1$ is a central extension of Carnot groups when equipped with the following structure.
    \begin{enumerate}
        \item The stratification of the Lie algebra is given by $  \hat{\mathfrak{g}}_1^{[j]} =  \mathfrak{h}_1^{[j]} \oplus  \varphi(V_1^{[j]})$ for $j \geq 1$.
        \item The inner product on $ \hat{\mathfrak{g}}_1$ is the direct sum of the inner products on  $ \mathfrak{h}_1$ and $\operatorname{im}( \varphi) \subset  \mathfrak{g}_2$.
    \end{enumerate}
    Moreover, the map $ \psi \colon  G_1 \to  \hat{G}_1$ given by
    \begin{equation*}
         \psi_{*}\colon \mathfrak{g}_1= \mathfrak{h}_1\oplus V_1\to \hat{\mathfrak{g}}_1= \mathfrak{h}_1\oplus\operatorname{im}( \varphi),\quad  \psi_{*}(X+Y) = X +  \varphi(Y),
	\end{equation*}
    is a surjective graded homomorphism.
\end{lemma}
\begin{proof}
    The fact that $ \hat{\rho}_1$ is a $2$-cocycle follows from the linearity of $ \varphi$. By \cref{lemma-homomorphism-central-extension}, the map $ \psi_{*}$ is a Lie algebra homomorphism.
    
    We observe that $\rho_1( \mathfrak{h}_1^{[1]}, \mathfrak{h}_1^{[j]} ) \subset V_1^{[j+1]}$ which immediately implies that $[\hat{\mathfrak{g}}_1^{[1]}, \hat{\mathfrak{g}}_1^{[j]}] = \hat{\mathfrak{g}}_1^{[j+1]}$ for $j \geq 1$. Thus (1) indeed defines a stratification. The requirements on the inner products are built-in to (2). The fact that $\psi_{*}$ is graded and surjective is clear. In conclusion, the central extension $\operatorname{im}( \varphi) \to  \hat{G}_1\to H_1$ by $ \hat{\rho}_1$ is a central extension of Carnot groups.
\end{proof}

\section{Abelian factors in contact lifts}\label{sec-structure-of-lifts}
In order to simplify later proofs, we show that any abelian Carnot group factors added by the central extensions are irrelevant for our existence results for contact lifts.

For this section, we consider central extensions of Carnot groups $ V_1\to G_1\to H_1$ and $ V_2\to G_2\to H_2$ and a smooth contact map $ f\colon U_1\to H_2$ from a domain $ U_1\subset H_1$.
Let $ W_1\subset V_1$ be the horizontal component of the central extension $ V_1\to G_1\to H_1$ so that the extended Carnot group can decomposed as a direct product of Carnot groups $ G_1 \simeq \tilde{G}_1\times  W_1$, where $\tilde{G}_1= G_1/ W_1$. We denote similarly on the codomain the horizontal component $ W_2\subset V_2$ and the quotient $\tilde{G}_2= G_2/ W_2$.

\subsection{Elimination of abelian factors from the codomain}\label{sec-abelian-factorization:codomain}

On the codomain, the contact lift condition imposes no restrictions on the abelian factor.

\begin{lemma}\label{lemma-lift-image-abelian-component}
    If there exists a smooth contact lift $ F\colon G_1\supset \pi_1^{-1}( U_1)\to G_2$ of $ f$, then there also exist a smooth contact lift $\tilde{ F}\colon \pi_1^{-1}( U_1)\to\tilde{G}_2$ of $ f$ and a smooth map $h\colon \pi_1^{-1}( U_1)\to  W_2$ such that $ F$ is the pointwise product $ F = \tilde{ F}\cdot h$.
    Conversely, if $\tilde{ F}\colon \pi_1^{-1}( U_1)\to\tilde{G}_2$ is a smooth contact lift of $ f$, then the pointwise product $ F = \tilde{ F}\cdot h$ is a smooth contact lift of $ f$ for every smooth map $h\colon \pi_1^{-1}( U_1)\to  W_2$.
\end{lemma}
\begin{proof}
    Since $ G_2$ is a direct product of the Carnot groups $\tilde{G}_2$ and $ W_2$, any map $ F\colon G_1\supset \pi_1^{-1}( U_1)\to G_2$ may be written as a pointwise product $ F = \tilde{ F}\cdot h$ for some maps $\tilde{ F}\colon \pi_1^{-1}( U_1)\to\tilde{G}_2$ and $h\colon \pi_1^{-1}( U_1)\to  W_2$.

    For a direct product of Carnot groups, the horizontal distribution is the sum of the horizontal distributions. Hence $ F$ is contact if and only if both $\tilde{ F}$ and $h$ are contact. Since $ W_2$ is an abelian Carnot group, the contact property for $h$ trivially holds, so $ F$ is contact if and only if $\tilde{ F}$ is contact.

    Finally, we observe that $ W_2\subset V_2\subset\ker \pi_2$, so $ \pi_2 \circ  F = \tilde{ \pi}_2\circ \tilde{ F}$. Hence $ F$ is a lift of $ f$ if and only if $\tilde{ F}$ is a lift of $ f$.
\end{proof}

\subsection{Elimination of abelian factors from the domain}
On the domain, any abelian factor in a contact lift can be quotiented away when the central extension on the codomain does not increase horizontal rank.

\begin{lemma}\label{lemma-lift-domain-abelian-component}
    If there exists a smooth contact lift $\tilde{ F}\colon G_1\supset \pi_1^{-1}( U_1)\to\tilde{G}_2$ of $ f$, then $\tilde{ F}(gx) = \tilde{ F}(g)$ for all $g\in \pi_1^{-1}( U_1)$ and $x\in W_1$, and the induced quotient map $\tilde{\tilde{ F}}\colon \tilde{G}_1\supset\tilde{ \pi}_1^{-1}( U_1)\to\tilde{G}_2$ is a smooth contact lift of $ f$.
    Conversely, if $\tilde{\tilde{ F}}\colon\tilde{ \pi}_1^{-1}( U_1)\to\tilde{G}_2$ is a smooth contact lift of $ f$, then composing with the quotient map $ \pi\colon  G_1\to\tilde{G}_1$ defines a smooth contact lift $\tilde{ F} = \tilde{\tilde{ F}}\circ \pi|_{ \pi_1^{-1}( U_1) }$ of $ f$.
\end{lemma}
\begin{proof}
    The converse claim is immediate since the quotient projection $ \pi\colon  G_1\to\tilde{G}_1$ is contact and the composition of contact maps is contact.

    For the other claim, suppose that we have a contact lift $\tilde{ F}\colon G_1\supset \pi_1^{-1}( U_1)\to\tilde{G}_2$ of $ f$. Fix $g\in \pi_1^{-1}( U_1)$ and $x\in W_1$, and consider the horizontal line segment 
    \begin{equation*}
        \gamma\colon[0,1]\to  \pi_1^{-1}( U_1)\subset G_1,\quad \gamma(t) = g\cdot tx.
    \end{equation*}
    Since $\tilde{ F}$ is a contact lift of $ f$, and by construction $\operatorname{rank}\tilde{G}_2=\operatorname{rank} H_2$, the horizontal curve $\tilde{ F}\circ\gamma$ in $\tilde{G}_2$ is the unique lift of the horizontal curve $ f\circ \pi_1\circ \gamma$ starting from the point $\tilde{ F}\circ\gamma(0)$. 
    However, the projection is the constant curve $ f\circ \pi_1\circ\gamma(t) =  f(g)$, so also the lift $\tilde{ F}\circ\gamma$ is a constant curve. Hence
    \begin{equation*}
        \tilde{ F}(gx) = \tilde{ F}\circ\gamma(1)=\tilde{ F}\circ\gamma(0)=\tilde{ F}(g).
    \end{equation*}
    Consequently, there exists a unique map $\tilde{\tilde{ F}}\colon \tilde{G}_1\supset\tilde{ \pi}_1^{-1}( U_1)\to\tilde{G}_2$ such that $\tilde{ F} = \tilde{\tilde{ F}}\circ \pi$.

    Since the quotient projection $ \pi\colon  G_1\to\tilde{G}_1$ is surjective, $\tilde{\tilde{ F}}$ is a lift of $ f$. The map $\tilde{\tilde{ F}}$ is contact since also the derivative of $\tilde{ F}$ factors through the quotient projection $ \pi\colon  G_1\to\tilde{G}_1$ and the horizontal distribution of the direct product $ G_1 \simeq \tilde{G}_1\times W_1$ contains the horizontal distribution of $\tilde{G}_1$.
\end{proof}

\section{Contact lifts via path lifting}\label{sec:lifts}
    This section has three subsections. The first subsection connects central extensions and horizontal lifts of curves. The second subsection rephrases the existence of contact lifts in terms of horizontal curves. The third subsection reformulates the results of the second subsection in terms of horizontal exactness.

\subsection{A potential of an extension cocycle}
    The subsequent lemma constructs a potential for a cocycle defining a central extension. 
\begin{lemma}\label{lemma-explift-of-horizontal-curve}
    Let $ V\overset{\iota}{\to} G\overset{ \pi}{\to} H$ be a central extension of Carnot groups by $ \rho$ such that $\operatorname{rank}( G)=\operatorname{rank}( H)$.
    There exists a 1-form $ \alpha\in\Omega^1( H; V)$ such that
    \begin{enumerate}[(i)]
        \item\label{lemma-explift-enum-alpha-is-potential} $d \alpha =  \rho$, and
        \item\label{lemma-explift-enum-exp-alpha-integral-gives-missing-component} if $ \gamma\colon[0,1]\to  G$ is a horizontal curve starting from $ \gamma(0)= e_{ G}$, then
        \begin{equation*}
            \log_{ G}( \gamma(1)) = \log_{ H}( \pi\circ \gamma(1)) + \int_{ \pi\circ \gamma} \alpha \in  \mathfrak{h}\oplus V =  \mathfrak{g}.
        \end{equation*}
    \end{enumerate}
\end{lemma}
\begin{proof}
    The direct sum decomposition of the vector space $ \mathfrak{g}= \mathfrak{h}\oplus V$ defines a projection $ \pi_{ V} \colon  \mathfrak{g}\to V$. Consider the map $x=  \pi_{ V} \circ\log_{ G} \colon  G\to  V$,
    which gives the $ V$-component of exponential coordinates in $ G$. The projection $ \pi_{ V}$ can also be viewed as a $ V$-valued 1-form on $ \mathfrak{g}$. Let $ \theta\in\Omega^1( G; V)$ be its left-invariant extension.
    We claim that 
    \begin{equation}\label{eq-cocycle-potential-implicit-exp-coord-definition}
        dx- \theta= \pi^* \alpha
    \end{equation}
    for some 1-form $ \alpha\in\Omega^1( H; V)$, and that this 1-form $ \alpha$ has the required properties.

    To see that such a form $ \alpha$ exists, it suffices to verify that $(dx- \theta)(Y) = 0$ for any left-invariant vector field $Y$ with $Y( e_{ G}) \in\ker \pi =  V$. Since $ V$ is central, we have
    \begin{equation*}
        \log_{ G}(\exp(X)\exp(tY)) = X + tY
        \quad\text{for any $X \in  \mathfrak{g}$.}
    \end{equation*}
    Consequently, at any point $g=\exp(X)$, we have
    \begin{equation*}
        (d\log_{ G})_{g}(Y(g)) = \frac{d}{dt}\left.\log_{ G}(g\exp(tY))\right\vert_{t=0} = Y.
    \end{equation*}
    By the chain rule, we deduce
    \begin{equation*}
        dx(Y(g)) = d(  \pi_{ V} \circ \log_{ G} )(Y(g)) =  \pi_{ V} \circ  d\log_{ G}(Y(g)) =  \theta(Y(g)),
    \end{equation*}
    proving the required identity $(dx- \theta)(Y) = 0$. Hence \eqref{eq-cocycle-potential-implicit-exp-coord-definition} defines a 1-form $ \alpha\in\Omega^1( H; V)$.
    
    Property \eqref{lemma-explift-enum-alpha-is-potential} for $ \alpha$ follows immediately from \eqref{eq-cocycle-potential-implicit-exp-coord-definition} and the observation that $d \theta = - \pi^* \rho$.

    For property \eqref{lemma-explift-enum-exp-alpha-integral-gives-missing-component}, we first observe that
    \begin{equation*}
        \log_{ G}( \gamma(1)) = \log_{ H}( \pi\circ \gamma(1)) + x( \gamma(1))
    \end{equation*}
    by the definition of the map $x$. By \eqref{eq-cocycle-potential-implicit-exp-coord-definition}, we have $dx= \theta +  \pi^* \alpha$. Since $ \gamma$ is horizontal, we have $ \theta(\dot{ \gamma})=0$ almost everywhere. The initial point of the curve is the identity, so
    \begin{equation*}
        x( \gamma(1)) = x( \gamma(1)) - x( \gamma(0))
        = \int_{ \gamma}dx = \int_{ \gamma} \pi^* \alpha = \int_{ \pi\circ \gamma} \alpha,
    \end{equation*}
    proving the claim.
\end{proof}

\begin{remark}
    An explicit expression for the 1-form $ \alpha$ in \cref{lemma-explift-of-horizontal-curve} can be calculated from the derivative of the exponential map. The resulting expression evaluated at a point $g=\exp(X)\in G$, for a left-invariant vector field $Y$, has the form
    \begin{equation*}
         \alpha_g(Y(g)) =  \rho(X,\zeta(\ad X)Y),
    \end{equation*}
    where $\zeta$ is the analytic function $\zeta(z) = \frac{1}{1-e^{-z}} - \frac{1}{z}$. Here, nilpotency of $\ad X$ allows us to interpret $\zeta(\ad X)$ using a finite part of the power series expansion of $\zeta$.
\end{remark}

For the rest of this section, we fix central extensions of Carnot groups $ V_1\to G_1\to H_1$ and $ V_2\to G_2\to H_2$ by 2-cocycles $ \rho_1$ and $ \rho_2$, respectively.
\subsection{Preservation of closed horizontal curves}
In this subsection, we connect the existence of contact lifts to a condition on closed horizontal curves.

Given a Carnot group $ H$, a point $h \in  H$, and a domain $ U\subset H$, we denote by $ \Gamma_{\mathrm{LIP}}(h, U)$ the collection of closed Lipschitz curves $ \gamma \colon [0,1] \rightarrow  U$ based at $ \gamma(0) = h =  \gamma(1)$.

For a central extension $ V \to  G \to H$ of Carnot groups by a 2-cocycle $ \rho \in \Omega^{2}( H;  V)$, we denote by $ \Gamma^{\rho}_{\mathrm{LIP}}(h, U) \subset  \Gamma_{\mathrm{LIP}}(h, U)$ the subcollection which admit a horizontal lift to $ G$ that is closed. By \cref{lemma-explift-of-horizontal-curve}, and representing $ G$ as a direct product $ G\simeq \tilde{ G} \times W$, where $ W\subset V$ is the horizontal component of the central extension $V \to G \to H$, the subcollection can be characterized as
\begin{equation*}
     \Gamma^{\rho}_{\mathrm{LIP}}(h, U) = \{ \gamma\in \Gamma_{\mathrm{LIP}}(h, U): \int_{ \gamma} \alpha = 0 \},
\end{equation*}
where $ \alpha \in \Omega^{1}(  H;  V )$ satisfies $d \alpha= \rho$. Note that the collection is independent of the specific potential since $\int_{ \gamma} dy=0$ for any closed curve $ \gamma\in \Gamma_{\mathrm{LIP}}(h, U)$ and any exact 1-form $dy\in\Omega^1( H; V)$.

\begin{lemma}\label{lem-lift-iff-loops-to-loops}
    Let $ U_1\subset H_1$ and $ \tilde{U}_1 \subset  \pi_1^{-1}(  U_1 )\subset G_1$ be domains. A smooth contact map $ f\colon U_1\to H_2$ admits a smooth contact lift $ F\colon \tilde{U}_1\to G_2$ if and only if 
    \begin{equation}\label{eq-loops-to-loops}
        (  f \circ  \pi_1 )\left(  \Gamma_{\mathrm{LIP}}(g, \tilde{U}_1) \right) \subset \Gamma_{\mathrm{LIP}}^{ \rho_2}( f( \pi_1(g)), H_2)
    \end{equation}
    for some $g \in  \tilde{U}_1$.
\end{lemma}
\begin{proof}
If $ F \colon  \tilde{U}_1 \rightarrow  G_2$ is a smooth contact lift of $ f \colon  U_1 \rightarrow  H_2$ and $ \gamma \in  \Gamma_{\mathrm{LIP}}(g, \tilde{U}_1)$ is a closed horizontal curve based at an arbitrary $g\in  \tilde{U}_1$, then $ F \circ  \gamma$ is a closed horizontal curve in $ G_2$ based at $ F(g)$ whose projection to $ H_2$ is the closed horizontal curve $ \pi_2\circ F \circ  \gamma =  f \circ  \pi_1 \circ  \gamma$ based at $ f(  \pi_1(g) )$. Thus \eqref{eq-loops-to-loops} follows.

For the converse direction, we first observe that by \cref{lemma-lift-image-abelian-component}, it suffices to consider the case $\operatorname{rank}( G_2)=\operatorname{rank}( H_2)$.
Suppose that \eqref{eq-loops-to-loops} holds for some $g \in  \tilde{U}_1$. We will construct the lift $ F$ by a standard path lifting argument seen, for example, in \cite[Theorem~5.3]{Capogna-Tang-1995-quasiconformal_mappings_heisenberg}. We formulate the argument using control theory. We refer the reader to \cite[Section 8]{Agrachev-Barilari-Boscain-2020-a-comprehensive-introduction} for background on the endpoint map.

Fix a basepoint $p \in  \pi_2^{-1}(  f( \pi_1(g)) )$.
We construct a smooth contact lift $ F \colon  \tilde{U}_1 \rightarrow  G_2$ of $ f$ with $ F(g) = p$. Recall that for every control $u \in L^{2}( [0,1],   \mathfrak{g}_1^{[1]})$, there exists a corresponding trajectory $ \gamma_u \colon [0,1] \rightarrow  G_1$ obtained as the unique solution of the Cauchy problem
\begin{align*}
    \dot{ \gamma}_u(t) = ( L_{ \gamma_{u}(t)} )_{*}u(t),
    \quad
     \gamma_u(0) = g.
\end{align*}
This defines the endpoint map 
\begin{equation*}
    \operatorname{End}_g\colon L^{2}( [0,1]; \mathfrak{g}_1^{[1]})\to G_1,\quad \operatorname{End}_g(u) =  \gamma_u(1),
\end{equation*}
which by \cite[Proposition~8.5]{Agrachev-Barilari-Boscain-2020-a-comprehensive-introduction} is smooth.
We will also consider the two analogously defined endpoint maps 
\begin{align*}
    \operatorname{End}_{ f( \pi_1(g))}&\colon L^2([0,1]; \mathfrak{h}_2^{[1]})\to H_2,\\
    \operatorname{End}_{p}&\colon L^2([0,1]; \mathfrak{g}_2^{[1]})\to G_2.
\end{align*}
We have $  \mathfrak{h}_2^{[1]}= \mathfrak{g}_2^{[1]}$, so these latter two endpoint maps have the same domain. We note that $\operatorname{End}_{ f( \pi_1(g))} =  \pi_2 \circ \operatorname{End}_{p}$ by construction. 

Let $W\subset L^2([0,1]; \mathfrak{g}_{1}^{[1]})$ be the open subset of controls whose trajectories satisfy $ \gamma_u([0,1])\subset \tilde{U}_1$. Since $ \tilde{U}_1$ is a domain, $\operatorname{End}_p(W) =  \tilde{U}_1$. The smooth contact map $ f\circ \pi_1\colon \tilde{U}_1\to H_2$ induces a smooth map $( f\circ \pi_1)_*\colon W\to L^2([0,1]; \mathfrak{h}_{2}^{[1]})$ given by
\begin{equation*}
    (( f\circ \pi_1)_*u)(t) = (L_{ f\circ \pi_1\circ \gamma_u(t)}^{-1})_*\frac{d}{dt}( f\circ \pi_1\circ \gamma_u(t)).
\end{equation*}
We define $\Phi \coloneqq \operatorname{End}_{p} \circ ( f\circ \pi_1)_* \colon W \rightarrow  G_2$. Assumption \eqref{eq-loops-to-loops} implies that the value $\Phi(u)$ depends only on $\operatorname{End}_{g}(u)$ for $u \in W$. Thus there exists a uniquely defined map $ F \colon  \tilde{U}_1 \rightarrow  G_2$ for which $ F \circ \left. \operatorname{End}_g \right\vert_{W} = \Phi$.

By a perturbation argument, we see that, for any $h\in \tilde{U}_1$, there exists a control $u \in W \cap \operatorname{End}_{g}^{-1}( h )$ such that the differential $(d\operatorname{End}_{g})_u$ has full rank. So, by the implicit function theorem, every $h\in \tilde{U}_1$ has an open neighborhood $W\subset \tilde{U}_1$ such that a smooth right inverse $\Psi\colon V\to W$ of $\operatorname{End}_g$ exists. Therefore $\left. F\right\vert_{W} = \Phi \circ \Psi$ is smooth. It follows that $ F$ is smooth. Moreover, $ \pi_2 \circ  F =  f \circ  \pi_1$. Indeed, if $ \gamma \colon [0,1] \to  \tilde{U}_1$ is a horizontal curve, then $ F \circ  \gamma$ is a horizontal lift of $ f \circ  \pi_1 \circ  \gamma$ by the construction, the equality $\operatorname{End}_{ f( \pi_1(g))} =  \pi_2 \circ \operatorname{End}_{p}$, and the assumption \eqref{eq-loops-to-loops}, and thus the lifting property follows. It also follows that $ F$ is contact. That is, $ F$ is a smooth contact lift of $ f$.
\end{proof}

\subsection{Horizontal exactness}
In this subsection, we reformulate the path lifting result, \Cref{lem-lift-iff-loops-to-loops}, using horizontal exactness. This leads to the following definition.
\begin{definition}
    Let $ U$ be an open set in a Carnot group and $ V$ a vector space. A $1$-form $ \omega \in \Omega^{1}(  U;  V )$ is \emph{horizontally exact} if there exists a smooth function $x\colon  U\to V$ such that the horizontal components of $dx$ and $\omega$ are identical.
\end{definition}

\begin{proposition}\label{prop-lift-iff-potential-pullback-lift-exact}
    Let $ U_1 \subset  H_1$ and $ \tilde{U}_1\subset G_1$ be domains with $ \pi_1( \tilde{U}_1)= U_1$. Let $ \alpha_2\in\Omega^1( H_2; V_2)$ be a potential of $ \rho_2$. A smooth contact map $ f\colon H_1\supset U_1\to H_2$ admits a smooth contact lift $ F\colon \tilde{U}_1\to G_2$ if and only if $( f\circ \pi_1)^* \alpha_2$ is horizontally exact in $ \tilde{U}_1$.
\end{proposition}
\begin{proof}
We consider a direct product $ G_2 \simeq \tilde{ G}_2 \times   V_2^{[1]}$ for $\mathrm{rank}( \tilde{ G}_2 ) = \mathrm{rank}(  H_2 )$. Note that the $2$-cocycle $\rho_2$ has a potential $ \alpha_2$ taking values in $( V_2^{[1]} )^{\perp} =  V_2^{[2]} \oplus \dots \oplus  V_2^{[s]}$ because $\rho_2$ has no $V_2^{[1]}$-components. Since two potentials differ by an exact term, it suffices to consider this specific potential. Using this observation, the claim reduces to the case $\mathrm{rank}(  G_2 ) = \mathrm{rank}(  H_2 )$ by \Cref{lemma-lift-image-abelian-component}. After this reduction, we may also consider the potential used in \cref{lemma-explift-of-horizontal-curve}.

Suppose $(  f \circ  \pi_1 )^{*} \alpha_2$ is horizontally exact in $ \tilde{U}_1$. That is, suppose that there exists a smooth $y \colon  \tilde{U}_1 \rightarrow  V_2$ such that the horizontal components of $dy$ coincide with the horizontal components of $(  f \circ  \pi_1 )^{*} \alpha_2$. Whenever $ \gamma \colon [0,1] \rightarrow  \tilde{U}_1$ is a closed horizontal curve, we deduce that
\begin{align*}
    0 = \int_{ \gamma} dy = \int_{  \gamma } (  f \circ  \pi_1 )^{*} \alpha_2 = \int_{  f \circ  \pi_1 \circ  \gamma }  \alpha_2.
\end{align*}
By \cref{lem-lift-iff-loops-to-loops}, $ f$ admits a smooth contact lift.

For the converse direction, suppose that $ F \colon  \tilde{U}_1 \rightarrow  G_2$ is a smooth contact lift of $f$.

Let $x\colon G_2\to  V_2$ be the $ V_2$-component of exponential coordinates on $ G_2$ and let $y=x\circ F$. By \eqref{eq-cocycle-potential-implicit-exp-coord-definition} and the lift condition, we have
\begin{equation}\label{eq-vertical-component-derivative}
    dy =  F^*dx 
    =  F^* \theta_2 +  F^* \pi_2^* \alpha_2
    =  F^* \theta_2 + ( f\circ \pi_1)^* \alpha_2,
\end{equation}
where $ \theta_2\in\Omega^1( G_2; V_2)$ is the left-invariant extension of the projection $ \mathfrak{g}_2\to V_2$; here the projection is understood as a $ V_2$-valued $1$-form.
Since $ F$ is contact and $\mathrm{rank}(  G_2 ) = \mathrm{rank}(  H_2 )$, $ F^* \theta_2$ has no horizontal component. Hence \eqref{eq-vertical-component-derivative} implies that $( f\circ \pi_1)^* \alpha_2$ is horizontally exact. The equivalence follows.
\end{proof}

\section{Structure of contact lifts}\label{sec:uniqueness-etc}
This section has three subsections. The first subsection concerns the uniqueness of lifts up to left-translations. The second subsection shows that the lifting problem in simply connected domains is local. The third subsection shows that a lift is a homomorphism in the fiber direction of the projection.

In this section, we consider central extensions of Carnot groups $ V_1\to G_1\to H_1$ and $ V_2\to G_2\to H_2$ by 2-cocycles $ \rho_1$ and $ \rho_2$, respectively.

\subsection{Uniqueness of lifts}
When the codomain central extension preserves horizontal rank, smooth contact lifts are unique up to a left-translation by an element of $ V_2$.
\begin{lemma}\label{lemma-uniqueness-of-lifts}
    Suppose $\operatorname{rank}( H_2)=\operatorname{rank}( G_2)$ and let $ U_1\subset  H_1$ be a domain.
    If $ F_1, F_2\colon G_1\supset \pi_1^{-1}( U_1)\to  G_2$ are smooth contact lifts of a smooth contact map $ f\colon H_1\supset U_1\to H_2$, then there exists $k\in V_2$ such that $ F = k F_2$.
\end{lemma}
\begin{proof}
    Since $\operatorname{rank}( H_2)=\operatorname{rank}( G_2)$, the lift condition $ \pi_2\circ F_1= f\circ \pi_1 =  \pi_2\circ F_2$ implies that the horizontal derivatives of the smooth maps $ F_1$ and $ F_2$ agree everywhere. Since the set $ U_1$ is open and connected, it is horizontally path connected, so it follows that $ F_1= k F_2$ for some constant element $k\in G_2$. The lift condition further implies that $k\in\ker \pi_2 =  V_2$.
\end{proof}

\subsection{Lifts in simply connected domains}
The lifting problem on simply connected domains is local in the sense formulated in the subsequent lemmas.

\begin{lemma}\label{lemma-vertical-invariance-of-horizontal-exactness}
    Let $ U_1 \subset  H_1$ and $ \tilde{U}_1\subset G_1$ be simply connected domains with $ \pi_1( \tilde{U}_1)= U_1$. Then a smooth contact map $ f\colon H_1\supset U_1\to H_2$ admits a smooth contact lift $ F\colon \tilde{U}_1\to G_2$ if and only if it admits a smooth contact lift $ F\colon \pi_1^{-1}( U_1)\to G_2$.
\end{lemma}
\begin{proof}
    By \cref{prop-lift-iff-potential-pullback-lift-exact}, it suffices to check that $( f\circ \pi_1)^* \alpha_2$ is horizontally exact in $ \tilde{U}_1$ if and only if it is horizontally exact in the full cylinder $ \pi_1^{-1}( U_1)$. Since both domains are simply connected, horizontal exactness is equivalent to being Rumin closed. The claim follows since the form $( f\circ \pi_1)^* \alpha_2$ is invariant under translations by $ V_1$ and being Rumin closed is a local property.
\end{proof}

\begin{lemma}\label{lemma-local-lift-to-global-lift}
    Let $U_1\subset  H_1$ be a simply connected domain and let $ f\colon H_1\supset U_1\to H_2$ be a smooth contact map.
    Suppose that, for every $g\in  U_1$, there exists an open neighborhood $ U_g\subset U_1$ for which the restriction $\left. f\right|_{ U_g}\colon U_g\to  H_2$ admits a smooth contact lift $ F_g\colon \pi_1^{-1}( U_g)\to  G_2$. Then the unrestricted map $ f\colon U_1\to H_2$ admits a smooth contact lift $ F\colon \pi_1^{-1}( U_1)\to  G_2$.
\end{lemma}
\begin{proof}
    By \cref{lemma-lift-image-abelian-component}, it suffices to consider the case when $\operatorname{rank}( H_2)=\operatorname{rank}( G_2)$.
    Let $h =  f \circ  \pi_1|_{  \pi_1^{-1}( U_1) }$ and let $ \alpha_2$ be a potential for $ \rho_2$. By our assumption, for every $g \in  U_1$, the form $h^{*} \alpha_2$ is horizontally exact in the domain $ \pi_1^{-1}(  U_{g} )$ by \cref{prop-lift-iff-potential-pullback-lift-exact}. In particular, the horizontal projection $\pi_{E_0}h^{*} \alpha_2$ is Rumin exact in $ \pi_1^{-1}(  U_g )$, and thus Rumin closed in $ \pi_1^{-1}(  U_1 )$ by a covering argument. Since the Rumin cohomology is isomorphic to the de Rham one and $ \pi_1^{-1}(  U_1 )$ is simply connected, it follows that $\pi_{E_0}h^{*} \alpha_2$ is Rumin exact. Equivalently, $h^{*} \alpha_2$ is horizontally exact. Reapplying \cref{prop-lift-iff-potential-pullback-lift-exact} proves that $ f$ admits a lift into $ \pi_1^{-1}( U_1)$. 
\end{proof}

\begin{remark}
    \cref{lemma-local-lift-to-global-lift} can also be proved using more classical topological arguments in the vein of {\v{C}}ech cohomology without invoking \cref{prop-lift-iff-potential-pullback-lift-exact}. In the subsequent argument, we lose no generality in assuming that $\{ U_{g} \}_{ g \in U_1 }$ is a good cover of $U_1$, i.e. the finite intersections of $\{ U_g \}_{ g \in U_1 }$ are contractible. It suffices for us that $U_{g} \cap U_{g'}$ is connected.
    
    \cref{lemma-lift-image-abelian-component} allows us to reduce to the case $\operatorname{rank}( H_2)=\operatorname{rank}( G_2)$. By the uniqueness of lifts given in \cref{lemma-uniqueness-of-lifts}, if $ U_g\cap U_{g'}\neq\emptyset$ for some $g,g'\in U_1$, then there exists an element $\phi(g,g')\in V_2$ such that $ F_g=\phi(g,g') F_{g'}$ on the intersection $ \pi_1^{-1}( U_g\cap U_{g'})$.
    Then, for any threefold intersection $ U_g\cap U_{g'}\cap U_{g''}\neq\emptyset$, we obtain a 1-cocycle type compatibility condition
    \begin{equation}\label{eq-threefold-intersection-transition-compatibility}
        \phi(g,g'') = \phi(g,g') + \phi(g',g'').
    \end{equation}

    Using \eqref{eq-threefold-intersection-transition-compatibility} as definition of $\phi(g,g'')$ when $\phi(g,g')$ and $\phi(g',g'')$ are already defined, we extend $\phi$ to a map $\phi\colon U_1\times U_1\to  V_2$ 
    using finite chains of open sets $ U_{g_i}$. Since $ U_1$ is simply connected, this extension is well defined, and by construction satisfies \eqref{eq-threefold-intersection-transition-compatibility} for all $g,g',g''\in U_1$. Then, by fixing a basepoint $p\in  U_1$, we may define a map
    \begin{equation*}
         F\colon \pi_1^{-1}( U_1)\to G_2,\quad
         F(x) = \phi(p, \pi_1(x)) F_{ \pi_1(x)}(x).
    \end{equation*}
    Condition \eqref{eq-threefold-intersection-transition-compatibility} implies that $ F$ is a left-translation of $ F_g$ by an element in $ V_2$ in each $ U_g$ and hence a smooth contact lift of $ f$.
\end{remark}

\subsection{The fiber component of a lift}
    When the horizontal ranks on the codomain side coincide, a smooth contact lift is a homomorphism in the fiber direction of the projection.
\begin{lemma}\label{lemma-lipschitz-lift-is-a-homomorphism-on-fibers}
    Let $ V_1\to G_1\to H_1$ and $ V_2\to G_2\to H_2$ be central extensions of Carnot groups such that $\operatorname{rank}( G_2)=\operatorname{rank}( H_2)$.
    Let $ U_1\subset H_1$ be a domain and $ F\colon \pi_1^{-1}( U_1)\to G_2$ a smooth contact lift of a smooth contact map $ f\colon U_1\to H_2$.
    Then there exists a Lie group homomorphism $ \Phi\colon  V_1\to V_2$ such that $ F(gk)= F(g) \Phi(k)$ for all $g\in \pi_1^{-1}( U_1)$ and $k\in V_1$.
\end{lemma}
\begin{proof}
    By the lift assumption, we have
    \begin{equation*}
         \pi_2\circ  F(gk) =  f\circ \pi_1(gk) =  f\circ \pi_1(g) =  \pi_2\circ F(g).
    \end{equation*}
    Thus, for each $g\in  \pi_1^{-1}(  U_1 )$ and $k\in V_1$, there exists some $\tilde{ \Phi}(g,k)\in V_2$ such that $ F(gk)= F(g)\tilde{ \Phi}(g,k)$. We claim that $\tilde{ \Phi}(g,k)$ is constant in $g\in \pi_1^{-1}( U_1)$. That is, there is a well defined map $ \Phi\colon  V_1\to V_2$ such that $\tilde{ \Phi}(g,k) =  \Phi(k)$ for $g\in  \pi_1^{-1}(  U_1 )$ and $k\in  V_1$. By assumption, $ U_1$ is connected, so it suffices to prove that $g\mapsto \tilde{ \Phi}(g,k)$ is locally constant for fixed $k\in V_1$. 

    Let $B'$ be a ball that is compactly contained in $\pi_{1}^{-1}( U_1 )$. Let $B$ be a ball with the same center and half the radius, and let $L$ be an upper bound for the operator norm of the horizontal differential of $F$ on $B'$. Let $B''$ be a left-translation of $B'$ by $k$. Since $f \circ \pi_1$ is invariant under the action of $V_1$, it holds that $L$ is an upper bound for the operator norm of the horizontal differential of $f \circ \pi_1$ on $B''$. Since $F$ is a contact lift of $f$ and $\operatorname{rank}( G_2 ) = \operatorname{rank}( H_2 )$, it follows that the operator norm of the horizontal differential of $F$ is also bounded by $L$ on $B''$. Thus $F$ is $L$-Lipschitz on the left-translation of $B$ by $k$. We tacitly use this in the subsequent argument.
    
    Let $g,h\in B$.
    Since the extensions are central, elements of $ V_1$ commute with elements of $ G_1$, and similarly for $V_2$. By left-invariance of the distance, we compute that
    \begin{align*}
        d(\tilde{ \Phi}(g,k),\tilde{ \Phi}(h,k)) &= d( F(g)\tilde{ \Phi}(g,k), F(g)\tilde{ \Phi}(h,k))
        \\&\leq d( F(g)\tilde{ \Phi}(g,k), F(h)\tilde{ \Phi}(h,k))
        \\&+ d( F(h)\tilde{ \Phi}(h,k), F(g)\tilde{ \Phi}(h,k))
        \\&= d( F(gk), F(hk))
        + d( F(h), F(g))
        \\&\leq Ld(gk,hk) + Ld(g,h)
        \\&= 2Ld(g,h).
    \end{align*}
    That is, we have shown that $g\mapsto \tilde{ \Phi}(g,k)$ is a Lipschitz map $B\to V_2\subset  G_2$. By the equal rank assumption, $ V_2$ does not contain any horizontal directions. Hence the only possible Lipschitz maps $B\to V_2$ are constants. Consequently, $g\mapsto\tilde{ \Phi}(g,k)$ is locally constant for $k \in V_1$.
    
    Next, we show that $ \Phi\colon  V_1\to V_2$ is a group homomorphism.
    Let $k_1,k_2\in  V_1$ and let $g\in  \pi_1^{-1}( U_1)$. Applying the definition of $\tilde{ \Phi}$ with the pair $(g,k_1k_2)$ shows
    \begin{equation*}
         F(gk_1k_2) =  F(g) \Phi(k_1k_2),
    \end{equation*}
    whereas applying the definition with the pairs $(gk_1,k_2)$ and $(g,k_1)$ yields
    \begin{equation*}
         F(gk_1k_2) =  F(gk_1) \Phi(k_2)
        =  F(g) \Phi(k_1) \Phi(k_2).
    \end{equation*}
    Canceling out $ F(g)$ implies that $ \Phi$ is a group homomorphism.

    Finally, we observe that $ \Phi(k) =  F(g)^{-1} F(gk)$ for any $g\in G_1$ and $k\in V_1$, so the smoothness of $ \Phi$ follows.
\end{proof}

\section{Lifts and the Rumin complex}\label{sec:rumin}
In this section, we prove \cref{thm-lift-iff-potential-pullback-rumin-differential}. Since the Rumin differential does not commute with the pullback, we need some preliminary results about the Rumin differential of a 1-form in the context of central extensions.
We recall that for a left-invariant $2$-cocycle $ \rho$, we have that $ \rho \in \mathrm{im}(d_0)^{\perp}$ implies $ \rho \in E_0 \cap E$, and $ \rho \in E_0$ implies $ \rho \in \mathrm{im}(d_0)^{\perp}$.

\begin{lemma}\label{lemm:piE:projection}
    Let $ V\overset{\iota}{\to} G\overset{ \pi}{\to} H$ be a central extension of Carnot groups given by a 2-cocycle $ \rho\in\Omega^2( H; V)$, and let $ \theta\in\Omega^1( G; V)$ be the left-invariant 1-form extending the orthogonal projection $ \mathfrak{g}= \mathfrak{h}\oplus V\to V$. Let $v_1, \dots, v_m$ be a basis of $ V$ and decompose $ \theta = \sum_{j=1}^mv_j \theta^j$ and $ \rho = \sum_{j=1}^mv_j \rho^j$. Suppose that $ \rho^1,\ldots, \rho^m$ are all in $E_0(H)$ and are pairwise orthogonal and that $| \rho^j| \in \{0,1\}$ for $1 \leq j \leq m$. Let $ U \subset  H$ be an open set. Then every $1$-form $ \omega\in E_{0}^1( U)$ satisfies
    \begin{equation*}
        \pi_{E}  \pi^{*} \omega
        =
         \pi^{*} \pi_{E} \omega
        +
        \sum_{ j = 1 }^{ m } 
        \left( \langle d_c \omega,  \rho^j \rangle \circ \pi\right)\pi_{E_0^{\perp}} \theta^j,
    \end{equation*}
    where $\pi_{E_0^{\perp}}$ is the orthogonal projection to $E_0^{\perp}(  \pi^{-1}(  U ) )$.
\end{lemma}
\begin{proof}
    Denote $ \tilde{U}\coloneqq  \pi^{-1}( U)$. By construction of the central extension, the left-invariant 1-forms of $ G$ are linear combinations of the pullbacks of left-invariant 1-forms in $ H$ and the 1-forms $ \theta^j$. As $ \pi^{*}  \rho = -d_0 \theta$, the image of $d_0\colon\Omega^1( \tilde{U})\to\Omega^2( \tilde{U})$ is the $\mathcal{C}^{\infty}( \tilde{U})$-module spanned by $ \pi^*d_0\Omega^1( U)$ and the forms $ \pi^* \rho^1,\ldots, \pi^* \rho^m$. 
    Recall that the inner product in $ \mathfrak{g}$ is such that $ \pi_{*} \colon  \mathfrak{g} \rightarrow  \mathfrak{h}$ is a submersion. It follows that for any $ \omega\in\Omega^1( \tilde{U})$, the 2-form $ \omega \wedge  \theta^j$ is in $ \operatorname{im}(d_0)^{\perp}$.

    Any 2-form $\kappa \in \Omega^{2}(  \tilde{U} )$ has a representation
    \begin{align*}
        \kappa
        =
        c  \pi^{*}\gamma
        +
        \sum_{i = 1}^{m} c_i  \pi^{*}\gamma_i \wedge  \theta^i
        +
        \sum_{ i < j } c_{i,j}  \theta^{i} \wedge  \theta^{j},
    \end{align*}
    for some $\gamma \in\Omega^2( U)$, $\gamma_i\in\Omega^1( U)$, and $c,c_i,c_{i,j}\in\mathcal{C}^{\infty}( \tilde{U})$. Since the last two sums are orthogonal to $\operatorname{im}(d_0)$, and by assumption the forms $ \rho^j$ are pairwise orthogonal, the orthogonal projection of $\kappa$ to $\operatorname{im}(d_0)$ satisfies
    \begin{equation}\label{eq-imd0-projection}
         \pi_{ \mathrm{im}(d_0) } \kappa
        =  \pi_{ \mathrm{im}(d_0) } (c \pi^*\gamma)
        =
        c  \pi^{*}(  \pi_{\mathrm{im}(d_0)}\gamma) 
        +
        \sum_{ j = 1 }^{ m }
            \langle c \pi^*\gamma,  \pi^{*} \rho^j \rangle
             \pi^{*} \rho^j.
    \end{equation}

    For 1-forms, we have $E_0^1 = \ker d_0$. Denote $\widetilde{ \theta}^{j} = \pi_{E_0^{\perp}} \theta^j$, for $1 \leq j \leq m$, so that we have $\widetilde{ \theta}^{j} = d_0^{-1}d_0 \theta^j = d_0^{-1}(- \pi^* \rho^j)$.
    From \eqref{eq-imd0-projection}, it follows that
    \begin{equation}\label{eq-d0-inverse}
        d_0^{-1} \kappa
        = d_0^{-1} (c \pi^*\gamma)
        =
        c  \pi^{*}( d_{0}^{-1}\gamma )
        -
        \sum_{ j = 1 }^{ m }
            c\left( \langle \gamma,  \rho^j \rangle \circ \pi\right)
            \widetilde{ \theta}^j.
    \end{equation}
    Let $ \omega\in \Omega^1( U)$ be a 1-form. To compute $\pi_E \pi^* \omega$ for the projection $\pi_E$ given by \eqref{eq-pi-E-formula}, we first compute the powers $D^k \omega$ for the operator $D=d_0^{-1}(d-d_0)$.
    Denote $c_j^0 = 0$ and 
    \begin{equation*}
        c_j^k = \langle dD^{k-1} \omega ,  \rho^j \rangle
    \end{equation*}
    for $1\leq j\leq m$ and $k\geq 1$.
    We claim that
    \begin{equation}\label{eq-D-power-formula}
        D^k \pi^* \omega =  \pi^*D^k \omega - \sum_{ j = 1 }^{ m } ( c_j^{k} \circ  \pi ) \widetilde{ \theta}^j
        \quad\text{for $k \geq 0$.}
    \end{equation}
    
    The claim clearly holds for $k = 0$, so suppose inductively that the claim has been verified for some $k \geq 0$. Since $ \pi^*$ commutes with $d$ and $d_0$, and since $(d-d_0)\widetilde{ \theta}^j=0$, we have
    \begin{equation*}
        ( d - d_0 ) D^k \pi^* \omega
        =
         \pi^{*} ( d - d_0 )D^k \omega
        -
        \sum_{ j = 1 }^{ m }  \pi^{*}( (d-d_0) c_j^{k} ) \wedge \widetilde{ \theta}^j.
    \end{equation*}
    Applying \eqref{eq-d0-inverse} for $\kappa = ( d - d_0 ) D^k \pi^* \omega$ gives
    \begin{equation*}
        D^{k+1} \pi^* \omega
        =
         \pi^{*} D^{k+1} \omega
        -
        \sum_{j = 1}^{m} \left(\langle ( d - d_0 )D^k \omega,  \rho^j \rangle\circ \pi\right)\widetilde{ \theta}^j.
    \end{equation*}
    Since $ \rho^j \in \operatorname{im}( d_0 )^{\perp}$, we have 
    \begin{equation*}
        \langle ( d - d_0 )D^k \omega,  \rho^j \rangle
        = \langle dD^k \omega,  \rho^j \rangle = c_j^{k+1}
    \end{equation*}
    and the claim \eqref{eq-D-power-formula} follows.

    Next, we consider the operator $P = \sum_{ k \geq 0 }( -1 )^{k} D^{k}$. Using \eqref{eq-D-power-formula}, we obtain
    \begin{equation}\label{eq-P-formula}
        P  \pi^{*} \omega
        =
         \pi^{*} P \omega
        +
        \sum_{ j = 1 }^{ m } 
        \left(\langle d P \omega,  \rho^j \rangle\circ \pi\right)\widetilde{ \theta}^j.
    \end{equation}

    Since $d_0^{-1}$ is the zero map for 1-forms, formula \eqref{eq-pi-E-formula} simplifies down to $\pi_E = I - P d_0^{-1} d$ for 1-forms. Applying \eqref{eq-d0-inverse} for $\kappa=d \pi^{*} \omega$ yields
    \begin{align}\label{eq-d-0-inverse-d}
        d_0^{-1} d  \pi^{*} \omega
        =
         \pi^{*} d_{0}^{-1} d \omega
        -
        \sum_{ j = 1 }^{ m }
        \left(\langle d \omega,  \rho^j \rangle\circ \pi\right)\widetilde{ \theta}^j.
    \end{align}
    We apply the operator $P$ to \eqref{eq-d-0-inverse-d} and then use \eqref{eq-P-formula} for the form $d_0^{-1}d \omega$. This implies that
    \begin{align*}
        Pd_0^{-1} d  \pi^{*} \omega
        =
         \pi^{*} P d_{0}^{-1}d \omega
        &+
        \sum_{ j = 1 }^{ m } 
        \left(\langle d P d_0^{-1} d \omega,  \rho^j \rangle\circ \pi\right)\widetilde{ \theta}^j
        \\&-        
        \sum_{ j = 1 }^{ m }
        P\left( \left(\langle d \omega,  \rho^j \rangle \circ \pi\right)\widetilde{ \theta}^j \right).
    \end{align*}
    In the earlier computation for the powers $D^k \pi^{*} \omega$, we already saw the fact that
    \begin{align*}
        D\left( \left(\langle d \omega,  \rho^j \rangle \circ \pi\right)\widetilde{ \theta}^j \right)
        =
        d_0^{-1}\left(  \pi^*( d - d_0 )\langle d \omega,  \rho^j \rangle  \wedge \widetilde{ \theta}^j\right)
        = 0,
    \end{align*}
    which implies that $P\left( \left( \langle d \omega,  \rho^j \rangle \circ \pi\right) \widetilde{ \theta}^j \right) = \left( \langle d \omega,  \rho^j \rangle \circ  \pi\right) \widetilde{ \theta}^j$. Therefore
    \begin{align*}
        Pd_0^{-1} d  \pi^{*} \omega
        &=
         \pi^{*} P d_{0}^{-1}d \omega
        +
        \sum_{ j = 1 }^{ m } 
        \left( \langle d( P d_{0}^{-1}d - I ) \omega,  \rho^j \rangle\circ \pi\right)\widetilde{ \theta}^j
        \\
        &=
        \pi^{*} P d_{0}^{-1}d \omega
        -
        \sum_{ j = 1 }^{ m } 
        \left( \langle d\pi_{E} \omega,  \rho^j \rangle\circ \pi\right)\widetilde{ \theta}^j.
    \end{align*}
    We conclude that
    \begin{equation*}
        \pi_{E}  \pi^{*} \omega
        =
         \pi^{*} \pi_{E} \omega
        +
        \sum_{ j = 1 }^{ m } 
        \left( \langle d\pi_{E} \omega,  \rho^j \rangle \circ \pi\right)\widetilde{ \theta}^j.
    \end{equation*}
    Recalling $ \rho^j\in E_0(H)$, we see that
    \begin{equation*}
        \langle d\pi_{E} \omega,  \rho^j \rangle = \langle \pi_{E_0}d\pi_{E} \omega,  \rho^j \rangle.
    \end{equation*}
    Since $d\pi_E=\pi_Ed$, by the definition \eqref{eq-rumin-differential} of the Rumin differential, we have $d_c \omega = \pi_{E_0}d\pi_{E} \omega$ whenever $ \omega \in E_{0}^{1}(  U )$, concluding the proof.
\end{proof}

In the following lemma, we apply \Cref{lemm:piE:projection} to horizontal exactness on simply connected domains.
\begin{lemma}\label{lemma-rumin-exactness}
Let $ V\overset{\iota}{\to} G\overset{ \pi}{\to} H$ be a central extension of Carnot groups by $\rho$, and let $v_1, \dots, v_m$ be a basis of $ V$ and decompose $ \rho = \sum_{j=1}^mv_j \rho^j$. Let $ \omega \in \Omega^{1}( U; W )$ be a $1$-form in a simply connected domain $ U \subset  H$ for a finite-dimensional vector space $W$. Then $ \pi^{*} \omega$ is horizontally exact in $ \pi^{-1}(  U )$ if and only if $d_c \omega = \sum_{j = 1 }^{ m } c_j \pi_{E_0} \rho^j$ in $ U$ for some constant vectors $c_j\in W$, $1\leq j\leq m$.
\end{lemma}
\begin{proof}
Consider a basis $w_1, \dots, w_l$ for $W$ and denote $ \omega = \sum_{ j = 1 }^{ l } w_j  \omega^{j}$. Then $ \omega$ is horizontally exact if and only if each $ \omega^{j}$ is horizontally exact for $1 \leq j \leq l$. Hence the general claim follows from the real-valued case. So we assume $W = \mathbb{R}$ from this point onwards.

Consider orthonormal and left-invariant $\widehat{ \rho}^1, \dots, \widehat{ \rho}^k$ spanning the same subspace as $\{ \pi_{E_0} \rho^1, \dots, \pi_{E_0} \rho^{m} \}$. In case $k < m$, let $\widehat{ \rho}^{j} = 0$ for $k+1 \leq j \leq m$. We denote $\widehat{ \rho} = \sum_{j = 1}^{m} v_j \widehat{ \rho}^j$. Consider the central extension of Lie groups $ V\to\widehat{ G}\overset{\widehat{ \pi}}{\to} H$ by $\widehat{ \rho}$ and the Lie group isomorphism from $ G$ to $\widehat{ G}$ obtained from \Cref{lemma-carnot-central-extension-isomorphism}. We equip $\widehat{ G}$ with a Carnot structure making the Lie group isomorphism an isomorphism of Carnot groups, cf. \Cref{remark-carnot-central-extension-isomorphism-tilting}. It follows that $ \pi^{*} \omega$ is horizontally exact if and only if $\widehat{ \pi}^{*} \omega$ is horizontally exact.

By definition, horizontal exactness means exactness for the Rumin differential. Since $\widehat{ \pi}^{-1}( U)$ is simply connected, it holds that $\widehat{ \pi}^* \omega$ is horizontally exact if and only if $d_c\widehat{ \pi}^* \omega=0$. Since $\pi_{E_0}\widehat{ \pi}^* \omega=\widehat{ \pi}^*\pi_{E_0} \omega$ as $\pi_{E_0}$ is the horizontal projection for $1$-forms, it follows that $d_c\widehat{ \pi}^{*} \omega = 0$ is equivalent to $d\pi_{E}\widehat{ \pi}^*\pi_{E_0} \omega = 0$.

Consider the functions
\begin{equation*}
    c_{j} \coloneqq \langle d_c \omega, \widehat{ \rho}^j \rangle,\quad 1\leq j\leq m.
\end{equation*}
Using \cref{lemm:piE:projection}, we find that
\begin{equation}\label{eq:derivedidentity}
    d\pi_E\widehat{ \pi}^*\pi_{E_0} \omega
    = \widehat{ \pi}^*\left( d\pi_E\pi_{E_0} \omega - \sum_{j=1}^{m} c_j\widehat{ \rho}^j\right) 
    + \sum_{j=1}^{m}\widehat{ \pi}^*dc_j\wedge\pi_{E_0^{\perp}}\widehat{ \theta}^j.
\end{equation}
The collection $\pi_{E_0^{\perp}}\widehat{ \theta}^1, \dots, \pi_{E_0^{\perp}}\widehat{ \theta}^k$ is linearly independent by the linear independence of $\widehat{ \rho}^1, \dots, \widehat{ \rho}^k$. Combining this with the equality \eqref{eq:derivedidentity}, we deduce that $d\pi_E\widehat{ \pi}^*\pi_{E_0} \omega = 0$ if and only if
\begin{equation}\label{eq:derivedidentity:onecomponent}
    d\pi_E\pi_{E_0} \omega = \sum_{j=1}^{m}c_j\widehat{ \rho}^j
    \quad\text{for constants $\{ c_j \}_{ j = 1 }^{ m }$.}
\end{equation}
Applying $\pi_{E_0}$ to \eqref{eq:derivedidentity:onecomponent} leads to the equivalent equality
\begin{align}\label{eq:derivedidentity:onecomponent:rumin}
    d_c  \omega = \sum_{j = 1}^{m} c_j \widehat{ \rho}^j
    \quad\text{for constants $\{ c_j \}_{ j = 1 }^{ m }$.}
\end{align}
Hence $ \pi^{*} \omega$ is horizontally exact if and only if \eqref{eq:derivedidentity:onecomponent:rumin} holds. Since the $\mathbb{R}$-linear spans of $\{ \widehat{ \rho}^1, \dots, \widehat{ \rho}^{m} \}$ and $\{ \pi_{E_0} \rho^{1}, \dots, \pi_{E_0} \rho^{m} \}$ are equal, the claim follows.
\end{proof}

\begin{proof}[Proof of \cref{thm-lift-iff-potential-pullback-rumin-differential}]
By \cref{prop-lift-iff-potential-pullback-lift-exact}, the smooth contact map $ f\colon U_1\to H_2$ admits a contact lift $ F\colon \pi^{-1}( U_1)\to G_2$ if and only if $ \pi^* f^* \alpha_2\in\Omega^1( \pi^{-1}( U_1); V_2)$ is horizontally exact. 

Let $v_1,\ldots,v_{m}$ be a basis of $ V_1$. Consider the decomposition $ \rho_1 = \sum_{j=1}^{m} v_j \rho^j_{1}$ of the 2-cocycle $ \rho_1\in\Omega^2( H_1; V_1)$. By applying \cref{lemma-rumin-exactness} to $ \pi^* f^* \alpha_2$, we find that a contact lift exists if and only if $d_c f^{*}  \alpha_2 = \sum_{ j = 1 }^{ m } c_j \pi_{E_0} \rho^j_{1}$ for some constant vectors $c_j\in V_2$, $1 \leq j \leq m$. Therefore  
\begin{equation*}
    L\colon  V_1\to V_2,\quad L(v_j)=c_j, \quad\text{for $j = 1,\dots,m$},
\end{equation*}
defines a linear map for which $L \circ \pi_{E_0} \rho_1 = d_c f^{*} \alpha_2$. The converse is immediate as well.
\end{proof}

\section{Lifts and Lie algebra cohomology}\label{sec-common-extension}
In this section, we prove \cref{thm-if-lift-cohomology-constant,thm-max-weight-implies-lift}. First, we compare the Pansu derivatives of a smooth contact map and its lift.

\begin{lemma}\label{lemma-common-derivative-extension-is-necessary}
    Let $ V_1\to G_1\to H_1$ and $ V_2\to G_2\to H_2$ be central extensions of Carnot groups, and let $ U_1\subset H_1$ be a domain.
    Let $ F\colon \pi_1^{-1}( U_1)\to G_2$ be a smooth contact lift of a smooth contact map $ f\colon U_1\to H_2$.
    Then there exists a graded linear map $ \varphi\colon  V_1\to V_2$ such that, for every $h\in U_1$, there exists a graded Lie group homomorphism $ \psi_{h}\colon  G_1\to  G_2$ for which the following diagram commutes:
	\begin{center}
		\begin{tikzcd}
			 V_1 \rar\dar{ \varphi}
            &  G_1 \rar\dar{ \psi_{h}}
            &  H_1\dar{ d_Pf(h)}
            \\
			 V_2 \rar
            &  G_2 \rar
            &  H_2
		\end{tikzcd}
	\end{center}
\end{lemma}
\begin{proof}
    Consider first the case where $\operatorname{rank}( G_2)=\operatorname{rank}( H_2)$. In this case, we may apply \cref{lemma-lipschitz-lift-is-a-homomorphism-on-fibers}, and obtain a homomorphism $ \Phi\colon V_1\to V_2$ such that $ F(gk) =  F(g) \Phi(k)$ for all $g\in  \pi_1^{-1}(  U_1 )$ and $k\in V_1$.

    Let $g\in \pi_1^{-1}(  U_1 )$.
    Computing the Pansu derivative in a direction $h\in  V_1$, we obtain
    \begin{align*}
         d_P{ F}(g)h 
        &= \lim_{\lambda\to 0^{+}} \delta_{1/\lambda}\Big( F(g)^{-1} F(g \delta_{\lambda}h)\Big)
        \\&= \lim_{\lambda\to 0^{+}} \delta_{1/\lambda}\Big( F(g)^{-1} F(g) \Phi( \delta_{\lambda}h)\Big)
        \\&= \lim_{\lambda\to 0^{+}} \delta_{1/\lambda}\Big( \Phi( \delta_{\lambda}h)\Big) \eqqcolon  \varphi h.
    \end{align*}
    Since the projections $ \pi_1$ and $ \pi_2$ are homomorphisms, we have $ \pi_2\circ d_P{ F}(g) =  d_P{ f}( \pi_1 g)\circ \pi_1$. This gives us the commutative diagram
    \begin{center}
		\begin{tikzcd}
             V_1 \rar\dar{ \varphi}
            &  G_1\rar{ \pi_1}\dar{ d_P{ F}(g)}
            &  H_1\dar{ d_P{ f}( \pi_1 g)}
            \\
             V_2 \rar
            &  G_2\rar{ \pi_2}
            &  H_2
		\end{tikzcd}
	\end{center}
    Since $ \varphi$ is a graded linear map independent of the point $g$, the claim is proved in the equal rank case.

    If instead $\operatorname{rank}( G_2)>\operatorname{rank}( H_2)$, then, as in \cref{lemma-lift-image-abelian-component}, we may decompose $ G_2$ as a direct product of Carnot groups $ G_2\simeq\tilde{ G }_2\times  W_2$, with $ W_2\subset  V_2$ and $\operatorname{rank}( \tilde{G}_2 ) = \operatorname{rank}(H_2)$, and consider the reduced central extension $ V_2/ W_2\to\tilde{ G }_2 \to  H_2$ and reduced lift $\tilde{ F}= \pi\circ F\colon \pi_1^{-1}( U_1)\to\tilde{G}_2$, where $ \pi\colon  G_2\to\tilde{G}_2$ is the quotient projection.

    By construction $\operatorname{rank}(\tilde{G}_2)=\operatorname{rank}( H_2)$, so we may apply the previous argument to the reduced lift $\tilde{ F}$. We obtain the commutative diagram
    \begin{center}
		\begin{tikzcd}
             V_1 \rar\dar{\tilde{ \varphi}}
            &  G_1\rar\dar{ d_P{\tilde{ F}}(g)}
            &  H_1\dar{ d_P{ f}( \pi_1 g)}
            \\
             V_2/ W_2 \rar\dar[hook]
            & \tilde{G}_2\rar\dar[hook]
            &  H_2\dar{\operatorname{id}}
            \\
             V_2 \rar
            &  G_2\rar
            &  H_2
		\end{tikzcd}
	\end{center}
    The composition of $\tilde{ \varphi}\colon  V_1\to V_2/ W_2$ and the inclusion $ V_2/ W_2\to V_2$ gives the required graded linear map $ \varphi$.
\end{proof}

\begin{proof}[Proof of \cref{thm-if-lift-cohomology-constant}]
    By \cref{lemma-common-derivative-extension-is-necessary}, there exists a graded linear map $ \varphi\colon V_1\to V_2$ and a family of graded homorphisms $\{  \psi_{h} \}_{ h \in  U_1 }$ with the commutative diagram
    \begin{center}
		\begin{tikzcd}
             V_1 \rar\dar{ \varphi}
            &  G_1\rar\dar{ \psi_{h}}
            &  H_1\dar{ d_P{ f}(h)}
            \\
             V_2 \rar
            &  G_2\rar
            &  H_2
		\end{tikzcd}
	\end{center}
    By \cref{lemma-homomorphism-central-extension}, such a diagram implies that the corresponding Lie algebra homomorphisms satisfy 
    \begin{equation}\label{eq-lie-algebra-pansu-pullback-identity}
         \varphi\circ \rho_1 -  d_P{ f}(h)^* \rho_2 = d_0\mu_h
    \end{equation}
    for some linear map $\mu_h\colon \mathfrak{h}_1\to V_2$.
    
    Defining $ \omega\in\Omega^1( H_1; V_2)$ by $ \omega = d_0^{-1}(  \varphi\circ \rho_1 -  { f}_P^* \rho_2 )$, the identity \eqref{eq-lie-algebra-pansu-pullback-identity} may be restated as
    \begin{equation*}
         \varphi\circ \rho_1 -  f^{*}_{P} \rho_2 = d_0 \omega,
    \end{equation*}
    thereby proving the claim.
\end{proof}

    \Cref{thm-if-lift-cohomology-constant} and the following lemma allow us to reduce the existence of smooth contact lifts to a simpler situation.
\begin{lemma}\label{lemma-existence-of-lift-in-intermediate-extension-implies-lift}
	Let $ V_1\to G_1\to H_1$ and $ V_2\to G_2\to H_2$ be central extensions of Carnot groups by $2$-cocycles $ \rho_1$ and $ \rho_2$, respectively, and let $ f\colon H_1\supset U_1\to H_2$ be a smooth contact map, where $ U_1$ is an open set. Suppose that $ \varphi\colon V_1\to V_2$ is a graded linear map. Consider the central extension of Carnot groups $\operatorname{im}( \varphi)\to  \hat{G}_1 \to  H_1$ by $ \hat{\rho}_1 =  \varphi \circ  \rho_1$ and the graded homomorphism $ \psi \colon  G_1 \to  \hat{G}_1$ from \Cref{lemm:centralextension:morphism}. Then, if $ f$ admits a smooth contact lift $\hat{ F}\colon  \hat{G}_1\supset  \hat{\pi}_1^{-1}( U_1)\to  G_2$, the map $ F = \hat{ F}\circ  \psi|_{ \pi_1^{-1}( U_1)}$ is a smooth contact lift of $ f$.
\end{lemma}
    The proof is direct from the definitions. We are ready to prove \Cref{thm-max-weight-implies-lift}.

\begin{proof}[Proof of \cref{thm-max-weight-implies-lift}]
	By assumption, $ { f}_P^* \rho_2 =  \varphi\circ \rho_1 + d_0 \omega$ for some $ \omega \in \Omega^{1}(  U_1;  V_2 )$ and some graded linear map $ \varphi\colon V_1\to V_2$. By applying \cref{lemma-existence-of-lift-in-intermediate-extension-implies-lift}, we may assume that $ V_1 \subset  V_2$ and $ \varphi$ is the inclusion map. We suppress the notation for the inclusion map in the subsequent argument.
    
    We claim that it suffices to prove that 
	\begin{equation}\label{eq-commuting-pr}
		d_c  f_{P}^{*} \alpha_2 = \pi_{E_0}\pi_{E}  f_{P}^{*} \rho_2
	\end{equation}
    for a potential $ \alpha_2$ of $ \rho_2$.    
    
    We reduce to \Cref{thm-lift-iff-potential-pullback-rumin-differential} from \eqref{eq-commuting-pr} by claiming $d_c  f^{*} \alpha_2 = \pi_{E_0} \rho_1$. First, observe that the left-hand side of \eqref{eq-commuting-pr} can be replaced by $d_c f^{*} \alpha_2$ because the horizontal components of $ f_{P}^{*} \alpha_2$ and $ f^{*} \alpha_2$ coincide.

    We recall that $ \rho_1 = \pi_{E_0} \rho_1 + d_0\mu$ for a left-invariant $\mu \in \ker(d_0)^{\perp}$. We may also assume that $ \omega \in  \ker(d_0)^{\perp}$. We denote $\tilde{ \omega} = \mu +  \omega$. 
    We observe that in the identity 
    \begin{equation*}
         { f}_P^* \rho_2 = \pi_{E_0} \rho_1 + d_0\tilde{ \omega},
    \end{equation*}
    there cannot be any cancellation on the right-hand side, since $E_0\subset(\operatorname{im} d_0)^\perp$. By assumption, the weight of $f_P^{*}\rho_2$ is at least $\max\{ \operatorname{wt}(\tau): 0\ne \tau\in E_0^2 \}$, so the same must be true for $d_0\tilde{ \omega}$.
    
    On the other hand, we have $\pi_{E} \tilde{\omega} = 0$ since $\pi_E = I - P d_0^{-1} d$ for $1$-forms and $P$ is a left-inverse of $d_0^{-1} d|_{ \mathrm{im}( d_0^{-1} ) }$ by \cite[Lemma]{Rumin-1999-differential_geometry_on_cc_spaces}. Thus $\pi_Ed\tilde{ \omega} = d\pi_E\tilde{ \omega} = 0$. We may then rewrite
    \begin{equation*}
        \pi_{E_0}\pi_{E}d_0 \tilde{\omega} = -\pi_{E_0}\pi_{E}(d-d_0) \tilde{\omega}.
    \end{equation*}
    By the assumption that $\tilde{ \omega}\in (\ker d_0)^\perp$, the weights of $\tilde{\omega}$ and $d_0\tilde{\omega}$ are equal, and thus the weight of $( d- d_0 )\tilde{\omega}$ is strictly larger than $\max\{ \operatorname{wt}(\tau): 0\ne \tau\in E_0^2 \}$. Since $\pi_{E_0}\pi_{E}$ does not decrease weights, we find that $\pi_{E_0}\pi_{E}d_0 \omega = 0$. We obtain
    \begin{equation}\label{eq-commuting-pr-parttwo}
        \pi_{E_0}\pi_{E} f_{P}^{*} \rho_2 = \pi_{E_0} \rho_1,
    \end{equation}
    showing that once we prove \eqref{eq-commuting-pr}, we may apply \cref{thm-lift-iff-potential-pullback-rumin-differential} and \eqref{eq-commuting-pr-parttwo} to conclude the proof.

    For each $\varepsilon \in (0,1]$, let $ U_{1,\varepsilon} = \left\{ g \in  U_1 \colon d( g,  H_1 \setminus  U_1 ) > \varepsilon \right\}$. To prove \eqref{eq-commuting-pr}, we apply the center of mass mollification $ f_{\varepsilon} \colon  U_{1,\varepsilon} \rightarrow  H_2$ from \cite{Kleiner-Muller-Xie-2022-convolution_trick} to $ f$. The approximation theorem \cite[Theorem 1.3]{Kleiner-Muller-Xie-2022-convolution_trick} implies that
	\begin{equation}\label{eq-kmx-approximation}
		 f_{\varepsilon}^{*}\tau \wedge \nu
		\rightarrow
		 f^{*}_P\tau \wedge \nu
		\quad\text{in $L^{1}( \Omega^{n}( U_1) )$}
	\end{equation}
	whenever $\tau \in \Omega^{k}(  H_2 )$ and $\nu \in \Omega^{n-k}_c(  U_1 )$ satisfy
	\begin{equation*}
		\mathrm{wt}( \tau ) + \mathrm{wt}( \nu ) \geq  Q_{ H_1}.
	\end{equation*}        
	Fix a compactly supported form $\eta\in E_0^{n-2}( U_1)$. To prove \eqref{eq-commuting-pr}, it suffices to prove that
	\begin{equation}\label{eq-commuting-pr-integral-version}
		\int_{ U_1} d_c  f_{P}^{*} \alpha_2\wedge \eta = \int_{ U_1}\pi_{E_0}\pi_{E}  f_{P}^{*} \rho_2\wedge \eta.
	\end{equation}
	
	Observe that, for $\pi_E\eta\in E^{n-2}$, since $E \subset \mathrm{ker}( d_0^{-1} )$ and $d_0$ is the zero map on codimension one forms, we have $\pi_{E_0}\partial \pi_E\eta = \partial \pi_E\eta$. Expanding out the Rumin differential $d_c=\pi_{E_0}d\pi_E\pi_{E_0}$, this implies that $\partial_c\eta = \partial\pi_E\eta$. Let $\varepsilon>0$ be small enough that the support of $\eta$ is contained in $ U_{1,\varepsilon}$. Then we have
	\begin{equation}\label{eq-commuting-pr-for-mollification}
		\int_{ U_{1,\varepsilon}}  f_{\varepsilon}^{*} \alpha_2 \wedge \partial_c\eta
		=\int_{ U_{1,\varepsilon}}  f_{\varepsilon}^{*} \rho_2 \wedge \pi_E\eta.
	\end{equation}
	This implies that an analogue of \eqref{eq-commuting-pr-integral-version} holds for the mollification pullback $ f_{\varepsilon}^*$. We claim that the weight assumption on $ \rho_2$ allows us to apply the approximation result \eqref{eq-kmx-approximation} to take the limit as $\varepsilon\to 0$ and obtain \eqref{eq-commuting-pr-integral-version}.
	
	For the left-hand side, we do not need the weight assumption. Indeed, since any horizontal 1-form has weight 1, all forms in $E_0^1$ have weight 1; by duality all forms in $E_0^{n-1}$ have weight $ Q_{ H_1}-1$. Hence
    \begin{equation*}
		\lim\limits_{\varepsilon\to 0}\int_{ U_{1,\varepsilon}}  f_{\varepsilon}^* \alpha_2 \wedge \partial_c\eta
		= \int_{ U_1}  f^{*}_P \alpha_2 \wedge \partial_c\eta = \int_{ U_1} d_c f_P^* \alpha_2 \wedge \eta.
    \end{equation*}

	For the right-hand side, we use the weight assumption that $\operatorname{wt}( \rho_2)\geq \max\{ \operatorname{wt}(\tau): 0\ne \tau\in E_0^2 \}$. By duality, for $\eta\in E_0^{n-2}$, we have $\operatorname{wt}(\eta) \geq  Q_{ H_1}-\operatorname{wt}( \rho_2)$. Since $\pi_E\colon E_0\to E$ does not decrease weights, we conclude that
	\begin{equation*}
		\lim\limits_{\varepsilon\to 0}\int_{ U_{1,\varepsilon}}  f_{\varepsilon}^{*} \rho_2 \wedge \pi_E\eta
		= \int_{ U_1}  f^{*}_P \rho_2 \wedge \pi_E\eta
		= \int_{ U_1} \pi_{E_0}\pi_E { f}_P^* \rho_2 \wedge \eta.
	\end{equation*}
	Thus taking the limit of \eqref{eq-commuting-pr-for-mollification} as $\varepsilon\to 0$ proves \eqref{eq-commuting-pr-integral-version}.
\end{proof}

\section{Lifts in Lipschitz 1-connected spaces}\label{sec:lip1-connected}
In this section, we prove \cref{thm-trivial-homotopy-implies-lift}. We begin by formulating a version of Stokes' theorem for the Pansu pullback.
\begin{lemma}\label{lemma-stokes}
    Let $ H$ be a Carnot group and let $ \alpha, \omega\in \Omega^1( H; V)$. If $u\colon \mathbb{D} \to  H$ is Lipschitz with boundary trace $ \gamma \colon \mathbb{S}^1 \to  H$, then
    \begin{align}\label{eq:Pansupullback:Stokes}
        \int_{\mathbb{D}} u_{P}^{*}( d \alpha + d_0 \omega ) = \int_{ \gamma}  \alpha.
    \end{align}
\end{lemma}
\begin{proof}
We have that $u_{P}^{*}( d \alpha + d_0 \omega ) = u^{*}(d \alpha)$ almost everywhere. Indeed, $u_{P}^{*}(d \alpha) = u^{*}(d \alpha)$ follows from the fact that, in this case, the left-trivialized differential of $u$ coincides with the Pansu differential whenever the Pansu differential exists. The equality $u_{P}^{*}(d \alpha) = u_{P}^{*}( d \alpha + d_0 \omega )$ follows from the fact that $d_0$ commutes with the Pansu pullback, and the fact that $d_0$ is the zero map in the Euclidean space.

With these observations, Stokes' equation implies
\begin{align*}
    \int_{  \gamma }  \alpha
    =
    \int_{ \mathbb{D} } u^{*}(d \alpha).
\end{align*}
To see this, we may reduce the claim to exponential coordinates by recalling that $\log_{ H} \circ u$ has a Lipschitz extension to an open neighbourhood of $\overline{\mathbb{D}}$. Then the claimed equality follows, e.g., by considering $\mathbb{D}$ as an integral current $[\mathbb{D}]$ with oriented boundary $[\mathbb{S}^1]$ and using standard pushforward results for Lipschitz maps, see \cite[4.1.14]{Fed:69} or \cite[Section 2]{AK:00:current}. So \eqref{eq:Pansupullback:Stokes} follows.
\end{proof}

\begin{proof}[Proof of \cref{thm-trivial-homotopy-implies-lift}]
By assumption $ { f}_P^* \rho_2 =  \varphi\circ \rho_1 + d_0 \omega$ for some $ \omega \in \Omega^{1}(  U_1;  V_2 )$ and some graded linear map $ \varphi\colon V_1\to V_2$.  By applying \cref{lemma-existence-of-lift-in-intermediate-extension-implies-lift}, we may assume that $ V_1 \subset  V_2$ and $ \varphi$ is the inclusion map. We suppress the notation for the inclusion map in the subsequent argument.

Let $h\in  U_1$ and $r>0$ be such that the open ball $B( h, \tfrac{3\lambda r}{2} )$ is contained in $ U_1$. Since such balls cover the simply connected open domain $ U_1$, by \cref{lemma-local-lift-to-global-lift}, it suffices to show that there exists a smooth contact lift $ F\colon \pi_1^{-1}(B(h,r))\to G_2$ of $ f$. Furthermore, by \cref{lemma-vertical-invariance-of-horizontal-exactness}, it suffices to show that there exists a smooth contact lift $ F\colon \tilde{U}_1\to G_2$, for any smaller open subset $ \tilde{U}_1\subset \pi_1^{-1}(B(h,r))$ such that $ \pi_1( \tilde{U}_1)=B(h,r)$.

Fix a basepoint $g\in \pi_1^{-1}(h)$ and let $ \tilde{U}_1 = B( g, r )\subset G_1$ be the open ball. We will apply \cref{lem-lift-iff-loops-to-loops} to prove the existence of a lift, so let $ \gamma\in \Gamma_{\mathrm{LIP}}(g, \tilde{U}_1)$ be an arbitrary closed horizontal curve. We write $ \gamma$ as a concatenation $ \gamma = \sigma_1 \star \dots \star \sigma_m$, where each $\sigma_j \colon [0,1] \rightarrow  \tilde{U}_1$ has length $\ell( \sigma_j ) \leq r$. We form a new closed horizontal curve $ \gamma_j \colon [0,1] \rightarrow  \tilde{U}_1$ by first joining $g$ to $\sigma_j(0)$ and $\sigma_{j}(1)$ to $g$ by a curve of length $\leq r$ and concatenating these curves with $\sigma_j$ to form $ \gamma_j$. We lose no generality in assuming that the curve joining $\sigma_{j}(1)$ to $g$ and the curve joining $g$ to $\sigma_{j+1}(0)$ are inversions of one another. We prove that
\begin{align*}
    \int_{  f \circ  \pi_1 \circ  \gamma }  \alpha_2
    =
    \sum_{ j = 1 }^{ m } \int_{  f \circ  \pi_1 \circ  \gamma_j }  \alpha_2
    =
    \sum_{ j = 1 }^{ m } \int_{  \pi_1 \circ  \gamma_j }  \alpha_1
    =
    \int_{  \pi_1 \circ  \gamma }  \alpha_1
\end{align*}
for potentials $ \alpha_1$ and $ \alpha_2$ of $ \rho_1$ and $ \rho_2$, respectively. Observe that this suffices for the claim. Indeed, since $ \gamma\in \Gamma_{\mathrm{LIP}}(g, \tilde{U}_1)$, we have $\int_{  \pi_1 \circ  \gamma }  \alpha_1=0$, implying that condition \eqref{eq-loops-to-loops} holds. Then \cref{lem-lift-iff-loops-to-loops} implies that a smooth contact lift of $ f$ exists.

To finish, without loss of generality $ \gamma = \gamma_j$ and $\ell(  \gamma ) \leq 3r$. We may reparametrize it as a $\tfrac{3r}{4}$-Lipschitz curve $ \gamma \colon \mathbb{S}^1\to \tilde{U}_1$. By the Lipschitz $1$-connected assumption, there exists a $\tfrac{3\lambda r}{4}$-Lipschitz map $u\colon\mathbb{D}\to H_1$, with boundary trace $ \pi_1\circ  \gamma$, satisfying
\begin{align*}
    u( \mathbb{D} ) \subset B( h, \tfrac{3\lambda r}{2} ) \subset   U_1.
\end{align*}
By assumption, $ { f}_P^* \rho_2 =  \rho_1+ d_0 \omega$. From the chain rule, we obtain that
\begin{align*}
    (  f \circ u )_{P}^{*} \rho_2
    =
    u^{*}_{P}(  f^{*}_{P} \rho_2 )
    =
    u^{*}_{P}(  \rho_1 + d_0\omega )
\end{align*}
almost everywhere. \cref{lemma-stokes}, therefore, implies that
\begin{align*}
    \int_{  f \circ  \pi_1 \circ  \gamma }  \alpha_2
    &=
    \int_{\mathbb{D}} (  f \circ u )^{*}_{P} \rho_2
    =
    \int_{\mathbb{D}} u_{P}^{*}(  \rho_1 + d_0\omega ) = \int_{  \pi_1 \circ  \gamma }  \alpha_1.
\end{align*}
The proof is complete.
\end{proof}

\section{Examples}\label{sec:examples}

    In Examples \ref{ex:pathlifting}, \ref{ex:lagrangian-legendrian}, and \ref{rem-factorization-of-lifts-filiform}, we discuss connections between existing literature on contact lifts and our main results. \Cref{ex:largeweight} shows a simple application of our main results. Examples \ref{rem:windingmap} and \ref{rem:non-simply-connected} are concrete examples of the lifting problem. Lastly, \Cref{rem:contact-equation} reformulates the contact equations in terms of differential forms. 
    
    The following example shows that the lifting problem of horizontal curves can be understood as a special case of our main results.
\begin{example}\label{ex:pathlifting}
    Consider Carnot groups $ G$ and $ H$ with $\operatorname{rank}( G) = \operatorname{rank}( H)$, together with a graded homomorphism $ \pi \colon  G \to  H$ for which $ \pi_{*}$ is a submetry. The existence of lifts of horizontal curves $ \gamma\colon[0,1]\to  H$ to horizontal curves $\tilde{ \gamma}\colon[0,1]\to G$, with $ \pi \circ \tilde{ \gamma} =  \gamma$, can also be viewed as an application of \cref{thm-lift-iff-potential-pullback-rumin-differential}. 
    
    Indeed, decomposing the ideal $\ker \pi$ as $\ker \pi =   V^{[2]}\oplus\dots\oplus  V^{[s]}$ with $V^{[k]}\subset \mathfrak{g}^{[k]}$, we may consider $ G$ as an iterated central extension of Carnot groups $ G= G_s\to G_{s-1}\to\dots\to G_1= H$ with $\ker( G_{k}\to G_{k-1})\simeq   V^{[k]}$. Since $d_c=0$ for 1-forms on $\RR$, the criterion of \cref{thm-lift-iff-potential-pullback-rumin-differential} trivially holds for each central extension in the sequence, and we obtain the existence of lifts of horizontal curves.
\end{example}

    We next discuss the Lagrangian--Legendrian correspondence.
\begin{example}\label{ex:lagrangian-legendrian}
    Let $ \rho \in \bigwedge^{2}( \mathbb{R}^{2n})^{*}$ be the standard symplectic form $ \rho = \sum_{i = 1}^{n} dx_i \wedge dy_i$ for $z = (x,y) \in \mathbb{R}^{2n}$. The central extension $\mathbb{R} \rightarrow  G \rightarrow \mathbb{R}^{2n}$ by $ \rho$ is isomorphic to the $n$'th Heisenberg group $\mathbb{H}_n$. In Allcock's proof of the isoperimetric inequality in $\mathbb{H}_n$, for $n \geq 2$, Allcock lifted isotropic maps $ f \colon \mathbb{R}^2 \supset \mathbb{D} \rightarrow \mathbb{R}^{2n}$ to contact maps $ F \colon \mathbb{D}\rightarrow \mathbb{H}_{n}$ via a lifting construction similar to \Cref{sec:lifts}. Here, the isotropic condition is $ f_{P}^{*} \rho = 0$. It is also called the Lagrangian condition (see e.g.\ \cite{Schoen:Wolfson:01}). Since the Lie algebra differential $d_0$ is zero in Euclidean spaces, the Lagrangian condition is necessary for the existence of a lift $ F$ as above, recall \Cref{thm-if-lift-cohomology-constant}. The fact that it is sufficient follows either from \Cref{thm-trivial-homotopy-implies-lift} or \Cref{thm-max-weight-implies-lift}. We also note that Magnani extended Allcock's construction to the Allcock groups in \cite[Theorem 1.3]{Magnani-2010-contact-equations}. 
\end{example}

    If the cocycle defining the central extension on the image side has large enough weight, the lift condition holds automatically.
\begin{example}\label{ex:largeweight}
    Consider a central extension of Carnot groups $ V_2 \to  G_2 \to  H_2$ by $ \rho_2$ with $\operatorname{wt}( \rho_2) \geq s+2$ and a step $s$-Carnot group $H_1$. We claim that if $f \colon H_1 \supset U_1 \rightarrow H_2$ is a smooth contact map on a simply connected domain $U_1$, then there exists a smooth contact lift $F \colon  U_1 \rightarrow  G_2$ of $f$. We show this as a simple application of \Cref{thm-max-weight-implies-lift}.

    We first observe that $\mathrm{ker}(d_0 \colon \Omega^{2}( U_1 ) \rightarrow \Omega^{3}( U_1 ) )$ (resp. $E_0^2(U_1)$) has a left-invariant basis of pure weight elements whose maximal weight is at most $s+1$. To this end, given a $2$-cocycle $\rho$, the cocycle condition $\rho \in\ker (d_0)$ implies that if $i+j \geq s+2$, then $\rho(\mathfrak{h}_1^{[i]},\mathfrak{h}_1^{[j]})\subset\operatorname{span} \rho( \mathfrak{h}_1^{[i-1]}, \mathfrak{h}^{[j+1]}_1)$. Recursion leads to the conclusion $\rho(\mathfrak{h}_1^{[i]},\mathfrak{h}_1^{[j]})\subset\operatorname{span}\rho(\mathfrak{h}_1^{[1]},\mathfrak{h}_1^{[j+i-1]}) = \{0\}$. The first observation follows. To finish, it holds that $f_P^{*}\rho_2 = 0$ because $d_0$ commutes with the Pansu pullback and the weight of $f_{P}^{*}\rho_2$ is at least $s+2$. Consequently, the assumptions of \Cref{thm-max-weight-implies-lift} apply for the trivial extension $\{0\} \to  H_1 \to  H_1$. Thus there exists a smooth contact lift $F \colon  U_1 \rightarrow  G_2$ of $f$.
\end{example}

    For the following examples, we specialize the discussion to the first Heisenberg group $\mathbb{H}_1$ and the higher filiform groups. We recall that these groups are not Lipschitz $1$-connected (see e.g.\ \cite{Wenger-Young-2014-lipschitz-homotopy-groups-heisenberg}) and thus \Cref{thm-trivial-homotopy-implies-lift} is not applicable.
\begin{example}\label{rem-factorization-of-lifts-filiform}
    Consider the real filiform Carnot group $F^{s}$ of step $s$. The corresponding Lie algebra is given by
    \begin{align*}
        \mathfrak{f}^{s} = \mathrm{span}( X, Y ) \oplus \mathrm{span}( Z_2 ) \oplus \dots \oplus \mathrm{span}( Z_{s} ),
    \end{align*}
    where $[X, Y] = Z_2$ and $[X, Z_k] = Z_{k+1}$ are the only non-trivial relations. Consider the corresponding dual basis $X^{*}, Y^{*}, Z_2^{*}, \dots, Z_{s}^{*}$.
    
    Observe that $\mathbb{R} \to F^{s+1} \overset{ \pi_s}{\to} F^{s}$ is a central extension of Carnot groups by the cocycle $ \rho_s = X^{*} \wedge Z_{s}^{*}$. Here $F^{1}$ is isomorphic to $\mathbb{R}^2$ and $F^{2}$ is isomorphic to the first Heisenberg group $\mathbb{H}_1$. In fact, $\mathbb{R} \rightarrow F^{2} \overset{ \pi_1}{\to} \mathbb{R}^2$ is a central extension of Carnot groups by the volume form $dx \wedge dy$.

    Our results can be used to classify all smooth contact maps $ f \colon F^{s} \supset  U_1 \rightarrow F^{s}$, for simply connected $ U_1$, that lift to a smooth contact map $ F \colon F^{s+1} \supset  \pi_{s}^{-1}( U_1) \rightarrow F^{s+1}$. Indeed, we claim that it is necessary and sufficient that
    \begin{equation}\label{eq:liftingcondition:filiform}
         f_{P}^{*} \rho_s = \lambda  \rho_s \quad\text{for some $\lambda \in \mathbb{R}$.}
    \end{equation}
    Note that the weight of $ \rho_s$ (and thus also $ f_{P}^{*} \rho_s$) is $s+1$ while the maximum weight in the image of $d_0$ is $s$. Hence \eqref{eq:liftingcondition:filiform} is necessary (\Cref{thm-if-lift-cohomology-constant}) and sufficient (\Cref{thm-max-weight-implies-lift}) for the existence of a contact lift $ F$.

    When we identify $\mathbb{R}^2$ and $F^1$, the condition \eqref{eq:liftingcondition:filiform} is equivalent to $f$ having a constant Jacobian determinant, thereby relating to the results of Capogna--Tang \cite{Capogna-Tang-1995-quasiconformal_mappings_heisenberg} and Balogh--Hoefer-Isenegger--Tyson \cite{Balogh-Hoefer-Isenegger-Tyson-2006-horizonta_fractals_heisenberg}.

    The higher step case has also been considered previously. Indeed, comparing \eqref{eq:liftingcondition:filiform} to \cite[Theorem 1.1]{Xie-2015-quasiconformal_maps_on_filiform_groups} (see also Warhust \cite{War-contact-qc-real-filiform}) leads to the following: if $s \geq 3$ and $ f \colon F^{s} \rightarrow F^{s}$ is bi-Lipschitz (and smooth), it is an iterated lift of a (smooth) bi-Lipschitz map $\mathbb{R}^2 \to \mathbb{R}^2$. By the constructions of Korányi--Reimann \cite{koranyi-reimann-1985-qc-maps-on-the-heisenberg-group,koranyi-reimann-1995-foundations-for-the-theory-of-QC-maps-in-the-heisenberg-group}, the assumption $s \geq 3$ is necessary.
\end{example}

As an example related to \Cref{rem-factorization-of-lifts-filiform}, we consider a contact lift of a planar map to the first Heisenberg group that does not lift to the higher filiform groups.

\begin{example}\label{rem:windingmap}
We consider the standard \emph{winding map} $ f \colon  U \to  U$ (of degree $k \geq 2$) for $ U = \mathbb{R}^2 \setminus \{0\}$. Recall that, in polar coordinates, the winding map satisfies $(r,t) \mapsto (r, k t )$. More precisely, let $p\colon (0,\infty)\times \mathbb{R} \to  U$ and $q \colon (0,\infty) \times \mathbb{R} \to (0,\infty) \times \mathbb{R}$ be the maps $(r,t)\mapsto (r\cos(t),r\sin(t))$ and $(r,t)\mapsto (r,kt)$, respectively. Then $ f\circ p=p\circ q$.

The map $ f$ can be presented locally as the composition $p\circ q\circ p^{-1}$ of smooth maps and is thus smooth. Further, since
\[
\det (Dp)_{(r,t)} = \det \begin{pmatrix}
    \cos(t) & -r\sin(t) \\
    \sin(t) & r\cos(t)
\end{pmatrix} = r
\]
and $\det (Dq)_{(r,t)}=k$ for $r>0$ and $t \in \mathbb{R}$, it follows that $\det (D f)\equiv k$ in $ U$.

As the first Heisenberg group is a central extension of $\mathbb{R}^2$ by the volume form $dx \wedge dy$, the determinant being constant implies that there exists a lift, at least locally. In fact, we may define $ F \colon  \pi_1^{-1}(  U ) \rightarrow \mathbb{H}_1$ as $ F(x,y,z) = ( f(x,y),kz)$ in exponential coordinates. Below we justify the contact property of $ F$ and that $ F$ does not lift to a contact map from the filiform group $F^{3}$ into itself, cf. \Cref{rem-factorization-of-lifts-filiform}.

In exponential coordinates, a standard left-invariant frame is given by
\begin{align*}
    X &= \partial_x - \frac{y}{2} \partial_z, \quad Y = \partial_y + \frac{x}{2} \partial_z, \quad\text{and}\quad Z_2 = \partial_z,
\end{align*}
where the dual frame is
\begin{align*}
    X^{*} = dx, \quad Y^{*} = dy, \quad\text{and}\quad Z_2^{*} = dz - \frac{1}{2}( x dy - y dx ).
\end{align*}

In cylindrical coordinates $(r, t, z)$, we have
\begin{align*}
    X^{*} &= \cos(t) dr - r \sin(t) dt,
    \\
    Y^{*} &= \sin(t) dr + r \cos(t) dt, \quad\text{and}
    \\
    Z_2^{*} &= dz - \frac{1}{2}r^2 d t.
\end{align*}
A direct computation shows that $ F^{*}Z_2^{*} = k Z_2^{*}$, so $ F$ is contact as claimed.

The left-trivialized differential of $F$ is of block diagonal form by the contact equations and the fact that $F^{*}X^{*}( Z_2 ) = 0 = F^{*}Y^{*}(Z_2)$. Then, if $g = (x,y,z)$ in exponential coordinates, it holds that
\begin{align*}
    (D_P F)_{(x,y,z)}( a, b ) = ( D_P f )_{ (x,y) }( a ) + k b,
    \quad\text{for $(a,b) \in \mathbb{R}^2 \oplus \mathbb{R}$,}
\end{align*}
where we identify $\mathbb{R}^2$ as a subspace of $\mathbb{R}^3 = \mathbb{R}^2 \oplus \mathbb{R}$ through the linear map $a_1 \partial_1 + a_2 \partial_2 \mapsto a_1 X + a_2 Y + 0 \cdot Z_{2}$ for $a = ( a_1, a_2 )$. Therefore the equalities
\begin{align*}
     F^{*}_{P}Z_{2}^{*} =  F^{*}Z_{2}^{*} = k Z_{2}^{*},
\end{align*}
and
\begin{align*}
     F^{*}_{P} \rho_2
    &=
    (  F_{P}^{*}X^{*} ) \wedge (  F_{P}^{*}Z_2^{*} )
    =
    ( \cos( k t ) dr - r \sin( k t ) k d t )
    \wedge
    k Z_2^{*}
    \\
    &=
    k \cos(kt) dr \wedge dz
    -
    \frac{1}{2}
    r^2 \cos(kt) k dr \wedge dt
    -
    r\sin( k t ) k^2 dt \wedge dz
\end{align*}
hold. In contrast,
\begin{align*}
     \rho_2
    =
        \cos(t) dr \wedge dz
        -
        \frac{1}{2}
        r^2 \cos(t) dr\wedge dt
        -
        r \sin(t) dt \wedge dz.
\end{align*}
By comparing the coefficients of $ F^{*}_{P} \rho_2$ and $ \rho_2$, it follows that no point $x_0 \in  U$ has an open neighbourhood $W$ and $\lambda \in \mathbb{R}$ such that $ F_{P}^{*} \rho_2 = \lambda  \rho_2$ in $W$. This implies that $ F$ does not admit a contact lift to $F^{3}$, even locally.
\end{example}

    We next give an example of a contact lift from a non-simply connected domain that does not extend to the full cylinder.
\begin{example}\label{rem:non-simply-connected}
    For non-simply connected domains $ U$, the argument of \cref{lemma-vertical-invariance-of-horizontal-exactness} to extend a lift to the entire cylinder $ \pi^{-1}( U)$ cannot always work. For instance, on the annulus $A=B(0,2)\setminus \bar{\mathbb{D}} \subset\RR^2$, consider the map $ f\colon A \to \RR^2$, $ f(x)=x/\abs{x}$, and the Heisenberg group $\mathbb{H}_1$. As in \Cref{rem-factorization-of-lifts-filiform}, recall that we have a central extension $\RR\to \mathbb{H}_1 \overset{ \pi_1}{\to}\RR^2$.

    In exponential coordinates, the horizontal lift of the curve $t\mapsto (r\cos(t),r\sin(t))$ to $\mathbb{H}_1$ is the spiral $t \mapsto (r\cos(t),r\sin(t),\frac{1}{2}r^2t)$ for $r >0$. The spirals above and their vertical translations foliate a domain $ \tilde{U}$. More precisely, identify $\mathbb{H}_1$ with $\mathbb{R}^3$ through exponential coordinates. Then $ \tilde{U}$ is the image of the diffeomorphism
    \begin{align*}
        \Psi \colon (1,2) \times \RR \times \left(-\frac{\pi}{2},\frac{\pi}{2} \right) \rightarrow  \tilde{U},
    \end{align*}
    where $(r,t,s) \mapsto ( r \cos t, r \sin t, \frac{1}{2} r^2 t + s )$. By construction, $ \pi_1( \tilde{U}) = A$. Consider the map
    \begin{equation}\label{eq-radial-projection-lift}
         F\colon \tilde{U}\to \mathbb{H}_1,\quad  F \circ \Psi(r,t,s) = (\cos t,\sin t,\frac{1}{2}t).
    \end{equation}
    It holds that $ \pi_1 \circ  F =  f \circ  \pi_1|_{  \tilde{U} }$ and that $ F$ is smooth. Using the notation from \Cref{rem:windingmap} and the formula for $ F \circ \Psi$, we obtain that $\Psi^{*}(  F^{*}Z_{2}^{*} ) = (  F \circ \Psi )^{*}Z_{2}^{*} = 0$. As $\Psi$ is a diffeomorphism, this implies that $ F^{*}Z_{2}^{*} = 0$ and thus $ F$ is a smooth contact lift of $ f$.

    If $ F$ would extend to a contact lift of $ f$ on $ \pi_1^{-1}( A )$, by \Cref{lemma-lipschitz-lift-is-a-homomorphism-on-fibers}, there would exist $k \in \mathbb{R}$ such that $ F(x,y,z+a) =  F( x,y,z ) + k (0,0,a)$ whenever $(x,y,z+a), (x,y,z) \in  \tilde{U}$. Such a $k$ does not exist. Indeed, if $r \in (1,2)$, evaluating $ F$ at $(r, 0, 0 )$ and $(r, 0, s )$ for $0\ne s \in \left(-\frac{\pi}{2}, \frac{\pi}{2}\right)$ leads to $k = 0$. Evaluating $ F$ at $(r, 0, 0)$ and $(r,0, \pi r^2)$ leads to a contradiction with $k = 0$. Thus $ F$ does not admit such an extension.
\end{example}

    As a final remark, we prove how our methods can be used to rewrite the contact equations for a contact map.
\begin{remark}\label{rem:contact-equation}
    The potential $ \alpha_2$ of the central extension 2-cocycle $ \rho_2$ appearing in \cref{thm-lift-iff-potential-pullback-rumin-differential} and our other results is directly related to the contact equation. Consider a smooth map $ f\colon H\supset  U \to G$ between Carnot groups, and consider $ G= G_s$ as an iterated central extension of Carnot groups $  \mathfrak{g}^{[k+1}\to G_{k+1]}\to G_k$ by $ \rho_{k+1}$, starting from the abelianization $ G_1= G/[ G, G]\simeq \RR^r$.

    For each $k=2,\ldots,s$, consider exponential coordinates $ G_k\to  \mathfrak{g}^{[1]}\oplus\dots\oplus \mathfrak{g}^{[k]}$ and let $x_k\colon G_k\to  \mathfrak{g}^{[k]}$ be the degree $k$-component of the coordinate map. Then, for the cocycle $ \rho_{k}\in\Omega^2( G_{k-1};  \mathfrak{g}^{[k]})$, we may consider the potential $ \alpha_{k}\in\Omega^1( G_{k-1};  \mathfrak{g}^{[k]})$ as in \cref{lemma-explift-of-horizontal-curve}. That is, $dx_k =  \theta_k +  \pi_{k}^* \alpha_k$, where $ \theta_k$ is the left-invariant extension of the projection $  \mathfrak{g}^{[1]}\oplus\dots\oplus \mathfrak{g}^{[k]}\to  \mathfrak{g}^{[k]}$, and $ \pi_{k}\colon G_k\to G_{k-1}$ is the projection.
    
    Denote by $\tau_k\colon G\to G_k$ the projection map to the quotient $ G_k$ for $k = 1,\ldots,s$. Then $\tau_{k}\circ f\colon  U\to G_{k}$ is a lift of $\tau_{k-1}\circ f\colon  U\to G_{k-1}$ for $k = 2,\ldots,s$. Computing as in the proof of \cref{prop-lift-iff-potential-pullback-lift-exact}, we see that
    \begin{equation*}
        d(x_{k}\circ \tau_k\circ f) = (\tau_k\circ f)^* \theta_{k} + (\tau_{k-1}\circ f)^* \alpha_k.
    \end{equation*}
    If $ f$ is contact, so is each $\tau_k\circ f$. Then, for a horizontal vector $Y$, we have $(\tau_k\circ f)^* \theta_{k}(Y)=0$ for $k = 2,\ldots,s$. That is, if we denote $\tilde{x}_k=x_{k}\circ \tau_k\colon  G\to  \mathfrak{g}^{[k]}$ and $\tilde{ \alpha}_k = \tau_{k-1}^* \alpha_k\in\Omega^1( G;  \mathfrak{g}^{[k]})$, then
    \begin{equation}\label{eq-contact-equation-using-cocycle-potential}
        d(\tilde{x}_k\circ f)(Y) =  f^*\tilde{ \alpha}_k(Y),
        \quad\text{for every $k = 2,\ldots,s$ and $Y \in   \mathfrak{g}^{[1]}$}.
    \end{equation}
    Conversely, if \eqref{eq-contact-equation-using-cocycle-potential} holds for every horizontal vector $Y$, then 
    \begin{equation*}
        (\tau_k\circ f)^* \theta_{k}(Y)=0,
        \quad\text{for every $k = 2,\ldots,s$ and $Y \in   \mathfrak{g}^{[1]}$}.
    \end{equation*}
    This precisely means that $ f$ maps horizontal vectors to horizontal vectors, i.e., $ f$ is contact. Thus \eqref{eq-contact-equation-using-cocycle-potential} is a form of the contact equations.
\end{remark}

\bibliographystyle{amsalpha}
\bibliography{biblio}

\providecommand{\bysame}{\leavevmode\hbox to3em{\hrulefill}\thinspace}
\providecommand{\MR}{\relax\ifhmode\unskip\space\fi MR }
\providecommand{\MRhref}[2]{%
  \href{http://www.ams.org/mathscinet-getitem?mr=#1}{#2}
}
\providecommand{\href}[2]{#2}
\begin{thebibliography}{KMX21b}

\bibitem[ABB20]{Agrachev-Barilari-Boscain-2020-a-comprehensive-introduction}
Andrei Agrachev, Davide Barilari, and Ugo Boscain, \emph{A comprehensive
  introduction to sub-{R}iemannian geometry}, Cambridge Studies in Advanced
  Mathematics, vol. 181, Cambridge University Press, Cambridge, 2020, From the
  Hamiltonian viewpoint, With an appendix by Igor Zelenko.

\bibitem[AK00a]{AK:00:current}
Luigi Ambrosio and Bernd Kirchheim, \emph{Currents in metric spaces}, Acta
  Math. \textbf{185} (2000), no.~1, 1--80.

\bibitem[AK00b]{Ambrosio-Kirchheim-2000-rectifiable-sets}
\bysame, \emph{Rectifiable sets in metric and {B}anach spaces}, Math. Ann.
  \textbf{318} (2000), no.~3, 527--555.

\bibitem[All98]{Allcock:98}
D.~Allcock, \emph{An isoperimetric inequality for the {H}eisenberg groups},
  Geom. Funct. Anal. \textbf{8} (1998), no.~2, 219--233.

\bibitem[BHIT06]{Balogh-Hoefer-Isenegger-Tyson-2006-horizonta_fractals_heisenberg}
Zolt\'{a}n~M. Balogh, Regula Hoefer-Isenegger, and Jeremy~T. Tyson, \emph{Lifts
  of {L}ipschitz maps and horizontal fractals in the {H}eisenberg group},
  Ergodic Theory Dynam. Systems \textbf{26} (2006), no.~3, 621--651.

\bibitem[CC06]{capogna-cowling-2006-conformality-and-q-harmonicity}
Luca Capogna and Michael Cowling, \emph{Conformality and {$Q$}-harmonicity in
  {C}arnot groups}, Duke Math. J. \textbf{135} (2006), no.~3, 455--479.

\bibitem[CT95]{Capogna-Tang-1995-quasiconformal_mappings_heisenberg}
Luca Capogna and Puqi Tang, \emph{Uniform domains and quasiconformal mappings
  on the {H}eisenberg group}, Manuscripta Math. \textbf{86} (1995), no.~3,
  267--281.

\bibitem[dG07]{deGraaf-2007-classification_of_6d_nilpotent}
Willem~A. de~Graaf, \emph{Classification of 6-dimensional nilpotent {L}ie
  algebras over fields of characteristic not 2}, J. Algebra \textbf{309}
  (2007), no.~2, 640--653.

\bibitem[Fed69]{Fed:69}
Herbert Federer, \emph{Geometric measure theory}, Die Grundlehren der
  mathematischen Wissenschaften, Band 153, Springer-Verlag New York Inc., New
  York, 1969.

\bibitem[FT15]{Franchi-Tripaldi-2015-diff_forms_in_carnot_after_rumin}
B.~Franchi and F.~Tripaldi, \emph{Differential forms in {C}arnot groups after
  {M}. {R}umin: an introduction}, Topics in Mathematics, Bologna
  (A.~Bonfiglioli, R.~Fioresi, and A.~Parmeggiani, eds.), Quaderni
  dell’Unione Matematica Italiana, vol.~55, Unione Matematica Italiana, 2015.

\bibitem[Gro96]{gromov-1996-carnot-from-within}
Mikhael Gromov, \emph{Carnot-{C}arath\'eodory spaces seen from within},
  Sub-{R}iemannian geometry, Progr. Math., vol. 144, Birkh\"auser, Basel, 1996,
  pp.~79--323.

\bibitem[KMX20]{Kleiner-Muller-Xie-2022-convolution_trick}
Bruce {Kleiner}, Stefan {Muller}, and Xiangdong {Xie}, \emph{{Pansu pullback
  and exterior differentiation for Sobolev maps on Carnot groups}}, arXiv
  e-prints (2020), arXiv:2007.06694.

\bibitem[KMX21a]{Kleiner-Muller-Xie-2021-pansu_pullback}
Bruce Kleiner, Stefan Muller, and Xiangdong Xie, \emph{{P}ansu pullback and
  rigidity of mappings between {C}arnot groups}, arXiv:2004.09271.

\bibitem[KMX21b]{Kleiner-Muller-Xie-2021-nonrigid_Carnot_groups}
Bruce {Kleiner}, Stefan {Muller}, and Xiangdong {Xie}, \emph{{Sobolev mappings
  between nonrigid Carnot groups}}, arXiv e-prints (2021), arXiv:2112.01866.

\bibitem[KR85]{koranyi-reimann-1985-qc-maps-on-the-heisenberg-group}
A.~Kor\'anyi and H.~M. Reimann, \emph{Quasiconformal mappings on the
  {H}eisenberg group}, Invent. Math. \textbf{80} (1985), no.~2, 309--338.

\bibitem[KR95]{koranyi-reimann-1995-foundations-for-the-theory-of-QC-maps-in-the-heisenberg-group}
\bysame, \emph{Foundations for the theory of quasiconformal mappings on the
  {H}eisenberg group}, Adv. Math. \textbf{111} (1995), no.~1, 1--87.

\bibitem[LD17]{LeDonne-2017-primer_on_carnot_groups}
Enrico Le~Donne, \emph{A primer on {C}arnot groups: homogenous groups,
  {C}arnot-{C}arath\'eodory spaces, and regularity of their isometries}, Anal.
  Geom. Metr. Spaces \textbf{5} (2017), no.~1, 116--137.

\bibitem[Mag04]{magnani-2004-unrectifiability-and-rigidity-in-stratified-groups}
Valentino Magnani, \emph{Unrectifiability and rigidity in stratified groups},
  Arch. Math. (Basel) \textbf{83} (2004), no.~6, 568--576.

\bibitem[Mag10]{Magnani-2010-contact-equations}
\bysame, \emph{Contact equations, {L}ipschitz extensions and isoperimetric
  inequalities}, Calc. Var. Partial Differential Equations \textbf{39} (2010),
  no.~1-2, 233--271.

\bibitem[OW11]{ottazzi-warhust-2011-contact-and-1-QC-on-carnot}
Alessandro Ottazzi and Ben Warhurst, \emph{Contact and 1-quasiconformal maps on
  {C}arnot groups}, J. Lie Theory \textbf{21} (2011), no.~4, 787--811.

\bibitem[Pan89]{Pansu-1989-metriques_de_carnoth_caratheodory}
Pierre Pansu, \emph{M\'etriques de {C}arnot-{C}arath\'eodory et
  quasiisom\'etries des espaces sym\'etriques de rang un}, Ann. of Math. (2)
  \textbf{129} (1989), no.~1, 1--60.

\bibitem[Rum99]{Rumin-1999-differential_geometry_on_cc_spaces}
Michel Rumin, \emph{Differential geometry on {C}-{C} spaces and application to
  the {N}ovikov-{S}hubin numbers of nilpotent {L}ie groups}, C. R. Acad. Sci.
  Paris S\'{e}r. I Math. \textbf{329} (1999), no.~11, 985--990.

\bibitem[Rum01]{Rumin-around-heat-decay}
\bysame, \emph{Around heat decay on forms and relations of nilpotent {L}ie
  groups}, S\'eminaire de {T}h\'eorie {S}pectrale et {G}\'eom\'etrie, {V}ol.
  19, {A}nn\'ee 2000--2001, S\'emin. Th\'eor. Spectr. G\'eom., vol.~19, Univ.
  Grenoble I, Saint-Martin-d'H\`eres, 2001, pp.~123--164.

\bibitem[SW01]{Schoen:Wolfson:01}
R.~Schoen and J.~Wolfson, \emph{Minimizing area among {L}agrangian surfaces:
  the mapping problem}, J. Differential Geom. \textbf{58} (2001), no.~1, 1--86.

\bibitem[War03]{War-contact-qc-real-filiform}
Ben Warhurst, \emph{Contact and quasiconformal mappings on real model filiform
  groups}, Bull. Austral. Math. Soc. \textbf{68} (2003), no.~2, 329--343.

\bibitem[WY10]{Wen:You:10}
Stefan Wenger and Robert Young, \emph{Lipschitz extensions into jet space
  {C}arnot groups}, Math. Res. Lett. \textbf{17} (2010), no.~6, 1137--1149.

\bibitem[WY14]{Wenger-Young-2014-lipschitz-homotopy-groups-heisenberg}
\bysame, \emph{Lipschitz homotopy groups of the {H}eisenberg groups}, Geom.
  Funct. Anal. \textbf{24} (2014), no.~1, 387--402.

\bibitem[Xie15]{Xie-2015-quasiconformal_maps_on_filiform_groups}
Xiangdong Xie, \emph{Quasi-conformal maps on model filiform groups}, Michigan
  Math. J. \textbf{64} (2015), no.~1, 169--202.

\end{thebibliography}
\end{document}